\documentclass[10pt]{amsart}

\usepackage{amsmath,amscd,amsfonts,amsthm,amssymb,url}
\usepackage{hyperref}
\usepackage[usenames,dvipsnames]{color}

 \DeclareFontFamily{OT1}{rsfs}{}
 
\DeclareFontShape{OT1}{rsfs}{n}{it}{<-> rsfs10}{}
\DeclareMathAlphabet{\mathscr}{OT1}{rsfs}{n}{it}

\DeclareMathOperator{\Hom}{Hom}

\newcommand{\GSp}{\mathrm{GSp}}

\newcommand{\lgp}{\mathcal{L}}
\newcommand{\Sur}{\mathrm{Sur}}

\DeclareMathOperator{\GL}{GL}
\DeclareMathOperator{\Sp}{Sp}

\newcommand{\baseR}{R}
\DeclareMathOperator{\Cl}{Cl}
\newcommand{\Mod}{\mathrm{Mod}}
\newcommand{\Rbar}{\bar{R}}

\newcommand{\Aone}{A(R)}
\newcommand{\Atwo}{A(R) +   \deg(U)}

\DeclareMathOperator{\Spec}{Spec}

\DeclareMathOperator{\Gal}{Gal}
\DeclareMathOperator{\Frob}{Frob}

\newcommand{\Hn}{\mathsf{Hn}}
\newcommand{\PHn}{\mathsf{PHn}}

\newcommand{\SConf}{\mathsf{Conf'}}
\newcommand{\SPConf}{\mathsf{PConf}}
\newcommand{\AConf}{\mathsf{Conf}}
\newcommand{\PreHn}{\mathsf{H}}

\DeclareMathOperator{\Conf}{Conf}
\DeclareMathOperator{\tConf}{\widetilde{Conf}}
\DeclareMathOperator{\Hur}{Hur}
\DeclareMathOperator{\CHur}{CHur}

\DeclareMathOperator{\Tor}{Tor}
\DeclareMathOperator{\Aut}{Aut}

\DeclareMathOperator{\Aff}{Aff}

\newcommand{\field}[1]{\mathbb{#1}}

\newcommand{\GG}{\mathcal{G}}
\newcommand{\G}{\mathbb{G}}
\newcommand{\Q}{\field{Q}}
\newcommand{\A}{\field{A}}

\newcommand{\Z}{\field{Z}}
\newcommand{\F}{\field{F}}

\newcommand{\modmod}{/\!\!/}

\newcommand{\Quads}{\mathfrak{S}}

\newcommand{\R}{\field{R}}
\newcommand{\C}{\field{C}}
\renewcommand{\P}{\field{P}}

\newcommand{\ra}{\rightarrow}

\newcommand{\OO}{\mathcal{O}}

\newcommand{\Tr}{\mbox{\textnormal Tr}}

\newcommand{\arc}{\mathcal{A}}

\newcommand{\tensor} {\otimes}

\newcommand{\bs}{\backslash}

\newcommand{\MM}{\mathcal{M}}
\newcommand{\Fqbar}{\overline{\mathbb{F}}_q}

\newcommand{\SHom}{\mathrm{Sur}}

\newcommand{\set}[1]{\{#1\}}
\newcommand{\Koszul}{\mathcal{K}}

\newcommand{\beq}{\begin{displaymath}}
\newcommand{\eeq}{\end{displaymath}}
\newcommand{\beqn}{\begin{equation}}
\newcommand{\eeqn}{\end{equation}}

\newcommand{\Diff}{\mathrm{Diff}}

\newcommand{\Jac}{\mathrm{Jac}}

\numberwithin{equation}{subsection}

\swapnumbers

\theoremstyle{plain}
\newtheorem{thm}[subsection]{Theorem}
\newtheorem{prop}[subsection]{Proposition}
\newtheorem{cor}[subsection]{Corollary}

\newtheorem{conj}[subsection]{Conjecture}
\newtheorem*{conj*}{Conjecture}

\newtheorem{lem}[subsection]{Lemma}

\newtheorem*{intro}{Theorem}

\theoremstyle{definition}
\newtheorem{defn}[subsection]{Definition}
\newtheorem*{defn*}{Definition}
\newtheorem*{prop*}{Proposition}
\newtheorem*{cor*}{Corollary}
\newtheorem*{lem*}{Lemma}

\theoremstyle{remark}
\newtheorem{rem}[subsection]{Remark}
\newtheorem*{rem*}{Remark}

\title[Homological stability and Cohen-Lenstra over function fields]{Homological stability for Hurwitz spaces and the Cohen-Lenstra conjecture over function fields}
\author{Jordan S. Ellenberg, Akshay Venkatesh and Craig Westerland}

 \dedicatory{Dedicated to Barry Mazur on the occasion of his 75th birthday}
\begin{document}

\maketitle

\begin{abstract}  
We prove a homological stabilization theorem for Hurwitz spaces:  moduli spaces of branched covers of the complex projective line.
This has the following arithmetic consequence:
let $\ell > 2$ be prime
and $A$ a finite abelian $\ell$-group. Then there exists $Q = Q(A)$ such that,
for $q$ greater than $Q$, a positive fraction of quadratic extensions of 
   $\mathbb{F}_q(t)$ have the $\ell$-part of their class group isomorphic to $A$.  
   \end{abstract}  
   \tableofcontents
\section{Introduction}

\subsection{The Cohen-Lenstra heuristics.}
Experimental evidence shows very clearly that class groups of number fields display interesting biases in their distribution.  For instance, class groups of quadratic imaginary fields are much more likely to contain a factor $\Z/9\Z$ than a factor $\Z/3\Z \times \Z/3\Z$.  Motivated by this and other examples,  Cohen and Lenstra conjectured in ~\cite{cohe:cohenlenstra}  that
a particular finite abelian group should occur as the class group of a quadratic imaginary field with frequency inversely proportional to its number of automorphisms. 

This leads, for instance, to the prediction that the probability that a quadratic imaginary field has class number indivisible by $3$ is
$$1 -\prod (1-3^{-i}) \sim 0.440 \dots$$

The initial motivation for the present paper was to study the Cohen-Lenstra heuristics
over function fields, i.e., finite extensions of $\mathbb{F}_q(t)$.   The result
quoted in the abstract can be stated more quantitatively as follows:

\begin{thm} \label{thm:CLweak}
Let $\ell > 2$ be prime
and $A$ a finite abelian $\ell$-group. 
Write $\delta^+$ (resp. $\delta^-$) for the upper density (resp. lower density) 
of imaginary\footnote{By ``imaginary" we mean ``ramified at $\infty$.''} quadratic extensions of $\mathbb{F}_q(t)$ for which the $\ell$-part of the class group is isomorphic to $A$.   Then $\delta^+(q)$ and $\delta^-(q)$ converge, as $q \rightarrow \infty$ with $q \neq 1 \pmod \ell$,
 to $\frac{ \prod_{i \geq 1} (1-\ell^{-i}) } {|\Aut(A)|}$. 
 
 \end{thm}

This is a corollary to Proposition \ref{cor:kluners} and Theorem~\ref{th:weakcl}.  When $q = 1 \pmod \ell$, the method of proof still works; for any fixed positive $\ell$-valuation of $q-1$, the proof yields a distribution which differs from the Cohen-Lenstra distribution.  (This is related to Malle's recent observation in \cite{mall:clroots} that the Cohen-Lenstra heuristics require modification when extra roots of unity are present in the base field.)   The description of this distribution in many cases is carried out in the Ph.D. thesis of Garton~\cite{gartonthesis}.

In particular, for $q > Q_0(\ell)$, a positive fraction of imaginary quadratic extensions of $\mathbb{F}_q(t)$ have class number divisible by $\ell$, and a positive fraction have class number indivisible by $\ell$.  The infinitude of quadratic extensions of $\F_q(t)$ with class number divisible by $\ell$ was previously known (\cite{pace:pacellireu}, \cite{byeo:indivisibility}) as was the corresponding result for indivisibility by $\ell$~\cite{frie:divisibility}, but in both cases without a positive proportion.

In a different direction, corresponding questions are understood if one studies quadratic field extensions of $\mathbb{F}_q(t)$ with {\em fixed discriminant degree} and lets $q \rightarrow \infty$; see \cite{Achter, Washington}. 

The essential ingredient in the proof of Theorem~\ref{thm:CLweak} is, perhaps surprisingly, a theorem in topology -- more precisely, a result on stable homology of Hurwitz spaces.

\subsection{Hurwitz spaces}

A {\em Hurwitz space} is a moduli space for $G$-covers of the punctured complex plane, where $G$ is a finite group.  A thorough definition of these spaces will be given in \S\ref{se:definitions}; here we content ourselves with a brief description.

Hurwitz spaces have vanishing higher homotopy groups; 
each component has fundamental group isomorphic to a subgroup of the Artin braid group $B_n$.
The group $B_n$ is generated by elements $\sigma_i, 1 \leq i \leq n -1$, subject to the relations:
\begin{eqnarray} 
\begin{aligned}  \label{artin-braid}  \begin{cases}  \ \sigma_i \sigma_{i+1} \sigma_i  = \sigma_{i+1} \sigma_i \sigma_{i+1},  & 1 \leq i \leq n-2 \\  \  \sigma_i \sigma_j  = \sigma_j \sigma_i,  & |i-j| \geq 2. \end{cases} \end{aligned} \end{eqnarray} A Hurwitz space can also be seen as the space of complex points of a {\em Hurwitz scheme} parametrizing branched covers of $\A^1$.  Consequently, the study of Hurwitz spaces lies at the interface of algebraic geometry, topology, and combinatorial group theory.   
 
The majority of the present paper involves only the topology of the Hurwitz space, not its algebro-geometric aspects (e.g., its definition as a scheme over a ring of integers).  We therefore start with a purely topological definition of Hurwitz space, in which we replace the complex plane
by the unit disc  $D = \{ (x,y) \in \R^2: x^2+y^2 \leq 1\}.$
 
The Eilenberg-Maclane space $K(B_n, 1)$ has the homotopy type of the configuration space 
 $\Conf_n D$, which parameterizes configurations of $n$ (distinct, unlabeled) points in the interior of the disc.    Fixing a point $*$ on the boundary of $D$, we define the {\em Hurwitz space} $\Hur_{G,n}$ to be the covering space of $\Conf_n D$ 
whose fiber above $\{P_1, \dots, P_n\}$ is the set of homomorphisms
$$\pi_1(D - \{P_1, \dots, P_n\}, *) \rightarrow G.$$ 
If $c \subset G$ is a conjugacy class,
we denote by $\Hur^c_{G,n}$ the open and closed subspace of $\Hur_{G,n}$ whose fiber over a point of $\Conf_n D$ is the set of homomorphisms sending a loop around each $P_j \ (1 \leq j \leq n)$ to the conjugacy class $c$.  

 The homotopy type of $\Hur_{G,n}$
is then that of the Borel construction $EB_n \times_{B_n} G^n$, where $B_n$ acts on $G^n$ through the {\em braiding action:}
\begin{equation} \label{braiding-action-def} \sigma_j: (g_1, \dots, g_n) \mapsto (g_1, \dots, g_{j-2},  g_{j-1}, g_j g_{j+1} g_j^{-1}, g_{j}, g_{j+2} \dots).\end{equation} 
 Similarly, the homotopy type of $\Hur_{G,n}^c$ is that of $EB_n \times_{B_n} c^n$.

\subsection{Stability of homology} \label{introhur}
The Hurwitz space is evidently not connected; for example, the braid group action (\ref{braiding-action-def}) preserves the subset of $c^n$ consisting of $n$-tuples with {\em full monodromy}, i.e., those whose elements generate the whole group $G$.

Hurwitz proved in \cite{hurwitz} that, when $G = S_d$ (the symmetric group on $d$ letters) and  $c$ is the conjugacy class of transpositions, the orbits of the $B_n$-action on
$$ \{ \mathbf{g} \in c^n: \mbox{$\mathbf{g}$ has full monodromy}\}$$
are determined by the \emph{boundary monodromy}; for $\mathbf{g} = (g_1, \dots, g_n)$, this is the product $g_1\cdots g_n \in S_d$.
  In geometric terms: 
the subspace $\CHur^c_{G,n}$ of $\Hur^c_{G,n}$, comprising covers with full monodromy, decomposes into a union of subspaces indexed by the boundary monodromy of a cover, each of which is connected\footnote{For analogous results for groups other than $S_d$, we refer the reader to the appendix of \cite{friedvolklein}.} for all sufficiently large $n$. %
 This result was used by Severi to establish that the moduli space $\MM_g$ of curves of genus $g$ is connected. 

By contrast with Hurwitz's connectivity result -- which we may think of as a statement about homology in degree zero -- very little is known about the higher homology of $\Hur_{G,n}^{c}$ or $\CHur^c_{G,n}$.  The main theorem of this paper is the following stabilization result for the homology of Hurwitz spaces.   We write $b_p(X)$ for $\dim H_p(X,\Q)$, the $p$th Betti number of a space $X$.

\begin{intro}  Let $G$ be a finite group and $c \subset G$ a conjugacy class such that
\begin{itemize}
\item $c$ generates $G$;
\item (non-splitting) For any subgroup $H \leqslant G$, the intersection of $c$ with $H$ is either empty or a conjugacy class of $H$.
\end{itemize}
Then there exist integers $A,B,D>0$ such that $b_p(\Hur^c_{G,n}) = b_p(\Hur^c_{G,n+D})$ whenever $n \geq Ap + B$.
\end{intro}

This theorem is proved as Theorem~\ref{th:stability} below, with constants $A,B,$ and $D$ which are explicitly computable in terms of the combinatorics of $G$ and $c$. It is the key input
in the proof of Theorem \ref{thm:CLweak}.  

We remark that not even the case $p=0$ is wholly obvious.  Indeed, it is false without the ``non-splitting'' condition -- for instance, if $G = S_4$ and $c$ is the conjugacy class of transpositions, $b_0(\Hur^c_{G,n})$ has rank at least $n+1$ coming from the components corresponding to $\mathbf{g} =(12)^i(34)^{n-i}$ for $i=0, \dots, n$.  Indeed,  the nonsplitting enters  into our argument
{\em only} to guarantee the validity of Lemma \ref{le:rurfinite}, whose concern is precisely the stability of the set of connected components, i.e. the $p=0$ case of homological stability.  When $(G,c)$ is {\em not} nonsplitting, the number of components
of $\Hur^c_{G,n}$ always grows without bound. 

Unfortunately, the non-splitting condition is very strong -- for instance, it is not satisfied for the case considered by Hurwitz ($G = S_d$ and $c$ the transpositions in $S_d$) unless $d=3$.  Fortunately, it is satisfied in the cases pertinent to the Cohen-Lenstra heuristics. 

In a sense, the fact that the spaces $\Hur^c_{G,n}$ are in general disconnected is one of the central difficulties that is overcome in this paper. It also marks a difference between our result and some (but not all) other results about stable homology, which we recall in section \S \ref{context} below.

Given this theorem, it is natural to ask whether the stable homology of Hurwitz spaces can be described explicitly.  Even the case $p=0$ (the description of the connected components of $\Hur_{G,n}^c$ for $n \gg 1$) is not obvious; the answer is governed by an argument of Conway, Parker, Fried, and V\"{o}lklein~\cite{friedvolklein}, the ideas of which we make crucial use of in this paper. 

In a sequel to the present paper, we will discuss the stable homology of Hurwitz spaces for $p > 0$.    As an example of the kind of results we expect, we propose the following conjectural generalization of Hurwitz's theorem to higher homology:

\begin{conj}
Suppose $G$ is a symmetric group on more than two letters, and $c$ the conjugacy class of transpositions.  Then for any 
$i \geq 0$, the map 
$$\CHur^c_{G,n} \ra \Conf_n D,$$
when restricted to a single component of the domain, induces an isomorphism of rational homology groups in degree $i$  for sufficiently large $n$.  
\label{co:vanishing}
\end{conj}

  (Recall, moreover, that so long as $n \geq 2$, the rational homology groups $H_j$ of $\Conf_n(D)$
vanish for $j > 1$ and are one-dimensional for $j=0, 1$).  This conjecture is motivated by -- and  implies a form of -- Malle's conjecture over function fields, which is to say that both the upper limit $\delta^+$ and the lower limit $\delta^-$ in Theorem~\ref{thm:CLweak} are equal to $1$.

\subsection{Some context} \label{context}

There is already a large  body of work in topology concerning homology stabilization
for certain ``geometrically natural'' sequences of manifolds with increasing dimensions. Examples include:

\begin{enumerate}
\item The configuration space $\Conf_n$ of $n$ points in the plane \cite{Cohen,Arnold}; 
\item The moduli space $\mathcal{M}_g$ of smooth projective curves of genus $g$ \cite{Harer} (more precisely, the moduli stack; its homology is taken to be orbifold homology) 
\item Classifying spaces of arithmetic groups, e.g., the space $\mathrm{BSL}(n, \mathbb{Z})$ \cite{Quillen, Borel}; 
\item \label{Segalexample} The space of holomorphic mappings $\mathrm{Maps}^d(\Sigma, X)$ of degree $d$ from a Riemann surface $\Sigma$ to a  suitable projective variety $X$ \cite{Segal, Milgram}.
\end{enumerate}

The Hurwitz spaces $\Hur_{G,n}$  have features in common with all of these examples.  On the one hand, their individual components are Eilenberg-Maclane spaces of type $K(\pi, 1)$, as are the first three examples --
in such cases, homological stability reduces to a question about group homology, for which
there are standard techniques (see \S \ref{subsec:proofhs}).  %

On the other hand,  we may also see $\Hur_{G,n}$ as parameterizing maps from a certain {\em orbifold} --  namely, a sphere on which $n$ points have finite cyclic inertia group -- to the classifying space $BG$, thereby relating it to the fourth example.\footnote{
We learned this point of view from Abramovich, Corti, and Vistoli~\cite{acv}, who use it to define an algebraic compactification of Hurwitz space as a space of {\em stable} maps from orbifolds to $BG$.}

The results of type 4 in the existing literature require the hypothesis that $X$ is simply connected.  The classifying space $BG$ is, of course, not simply connected; this has the effect that the spaces $\Hur_{G,n}$ we consider are typically not connected.  This feature turns out to be the source of all the technical difficulty in our paper.  %

\subsection{The proof of homological stability.}  \label{subsec:proofhs}

Our method to prove homological stability of Hurwitz spaces is based on the following (by now, standard) setup: 

Suppose that we are
given a sequence $G_1 \subset G_2 \subset \dots $ of
groups,  and, for each $n$, a highly connected $G_n$-simplicial complex $X_n$, such that the stabilizer of an $i$-simplex in $X_n$ is precisely $G_{n-i-1}$.
{\em Then the inclusions $G_n \rightarrow G_{n+1}$
tend to induce group homology isomorphisms.} We refer to a paper of Hatcher and Wahl \cite{HatcherWahl} for precise statements of this type.%

In our context, the pertinent complex is related to the work of 
Harer~\cite{Harer} on the homology of the moduli space of curves.

Throughout the method, however, the fact that the spaces $\Hur_{G,n}$ need not be connected
proves a difficulty. To handle this, we equip all the higher homology
groups with structures of module over the graded ring $R$ formed from the connected components of the Hurwitz spaces. We are then able to reduce all the difficulties to {\em purely homological} questions about $R$,  which are settled in \S \ref{Rhomology}. 
 
For the arithmetic applications, it is not sufficient to prove homological stability
of Hurwitz spaces; we need the a priori stronger statement of homological stability for Hurwitz schemes, moduli schemes over $\Spec \Z[\frac{1}{|G|}]$ whose complex points are isomorphic to $\Hur_{G,n}^c$.  This requires comparing the cohomology of generic and special fibers of the Hurwitz scheme, which is carried out in \S \ref{log}.

\subsection{Analytic number theory over function fields} \label{subsec:Analff} 

Many questions in analytic number theory over $\Z$, when transposed
to a function field setting, become questions of the following form:
\begin{equation}  \mbox{ Understand the asymptotics of $|X_n(\F_q)|$, as $n \rightarrow \infty$, } \end{equation}
where $X_n$ is an algebraic variety over $\F_q$ of dimension growing with $n$.  
For example, our analysis of the Cohen-Lenstra heuristics is based on 
the study of this question for $X_n$ a Hurwitz scheme. 
 We discuss some other examples in section~\ref{ss:stableexamples}.

The philosophy driving this paper can be summed up in the following slogan:
\begin{quote}
The quantity $|X_n(\F_q)| q^{-\dim X_n}$ should be expected to approach a limit as $n \ra \infty$ precisely when the varieties $X_n$ have stable homology.
\end{quote}
Of course, one can construct a sequence of varieties $\set{X_n}$ so that $|X_n(\F_q)| q^{-\dim X_n}$ approaches a limit but the homology of $X_n$ is not stable; the slogan is meant to apply just when $X_n$ is a ``natural" sequence of moduli spaces.

We now explain how one direction of the above slogan can be demonstrated in practice.  The Grothendieck-Lefschetz fixed point formula expresses $|X_n(\F_q)|$ 
in terms of the action of the Frobenius upon the compactly supported ({\'e}tale) cohomology of $X_n$:
\begin{equation} \label{Gro-Lef} |X_n(\F_q)| = \sum_{j} (-1)^j
\left( \mbox{trace of Frobenius acting on $H^j_{\mathrm{c}, et} (X_n  \times_{\F_q} \overline{\F_q}, \Q_{\ell})$}  \right)\end{equation}
For example, if $X_n=\mathbf{P}^n$ with $|X_n(\F_q)| = q^n + q^{n-1} + \dots + 1$ 
the term $q^j$ arises from $H^{2j}_{\mathrm{c}, et}$ on the right-hand side.  More generally, 
one expects that the dominant terms arise from the compactly supported cohomology in high degree, or, 
what is the same if $X_n$ is smooth, the usual (co)homology in low degree.  
This leads naturally to asking for some sense in which the low-degree (co)homology of $X_n$ is ``controlled.'' 

 For instance, suppose that the $X_n$ are smooth of dimension $n$ and geometrically irreducible for large $n$.
 Then the only nonvanishing terms on the right-hand sum of \eqref{Gro-Lef} occur in cohomological dimensions $j \leq 2n$, 
and the contribution of $j=2n$  is exactly $q^n$. 

 To bound the remaining terms, we suppose 
  that there exists a constant $C$ so that%
 \begin{equation} \label{mario} \dim H^i_{et}(X_n \times_{\F_q} \overline{\F_q}, \Q_{\ell})  \leq C^i, \ \ \mbox{ for all $n, i \geq 1$}.\end{equation}
The Deligne bounds \cite{Deligne} show that the eigenvalues of Frobenius on $H^{2n-i}_{\mathrm{c},et}(X_n \times_{\F_q} \overline{\F_q}, \Q_{\ell})$ are algebraic numbers all of whose complex eigenvalues are bounded above by $q^{n-i/2}$. Now, $\dim(H^{2n-i}_{\mathrm{c}, et}) \leq C^i$ 
  by \eqref{mario} and Poincar{\'e} duality, and so the trace of Frobenius on $H^{2n-i}_{\mathrm{c}, et}$
  is   bounded above by $q^{n} \cdot (C/\sqrt{q})^i$.  
Inserting this bound into the Lefschetz fixed point formula \eqref{Gro-Lef} 
to handle all terms with $j < 2n$,  we arrive at
\begin{equation} \label{lef-del} \left|  \frac{ \#X_n(\F_q)}{q^n}   -1   \right|   \leq \frac{1}{\sqrt{q}/C-1} \end{equation}
for all $q > C^2$.  In other words, $X_n$ has approximately $q^n$ points over $\F_q$, as one might naively guess.  (Indeed, some of the consequences of making this naive guess were discussed by the first and second authors in \cite{EV}, who, at the time, had no idea that the guess might under some circumstances be correct.)

How might one establish bounds of the form  \eqref{mario}? Suppose that we can establish the existence of an isomorphism $$H_i(X_n, \Q_{\ell}) \rightarrow H_i(X_{n+1},\Q_{\ell}),$$ for $i \leq n$.  (In fact, $i \leq An$ for any positive constant $A$ will be just as good for the type of application discussed in this paper.)   Here $H_i$ denotes the singular homology of the complex points of these varieties, equipped with the analytic topology.  One immediately obtains the bound
$\dim H_i(X_n, \Q_{\ell}) \leq \dim H_i (X_i, \Q_{\ell})$ for $i < n$.\footnote{If the isomorphisms arise from algebraic maps $X_n \ra X_{n+1}$ defined over $\F_q$, then one even has isomorphisms of etale cohomology groups compatible with Galois action.  We have not pursued this refinement in the present paper.}  In particular,  \eqref{mario}  would then follow from an upper bound of the form
 \begin{equation} \label{moon} \dim H^i_{et}(X_n \times_{\F_q} \overline{\F_q},  \Q_{\ell}) 
\leq C^n \end{equation}
for each $i$, which tends to be much easier: it can be checked given some a priori bound on the ``complexity'' of the variety $X_n$.    

So a theorem about homological stability can, in principle, be used to prove an asymptotic result in analytic number theory over function fields over finite fields.  We now present some examples in order to sketch the potential scope of this point of view.

\subsection{Stable homology and analytic number theory over function fields:  further examples}
\label{ss:stableexamples}

In this section we discuss some other problems that connect analytic number theory of function fields with the homology of a natural sequence of moduli spaces.

\begin{enumerate}
\item 
The number of squarefree integers in the interval $[X, 2X]$ is asymptotic to $\frac{X}{\zeta(2)}$. 
 Over the rational function field over $\F_q$, the corresponding question is: How many monic squarefree degree-$n$ polynomials are there in $\F_q[t]$?
 
Set $X_n = \Conf^n \A^1$, the configuration space of $n$ points on $\A^1$, or, equivalently, the space of monic squarefree polynomials of degree $n$.   In this case, one indeed has homological stability \cite{Arnold}:  the homology of $X_n$ with $\mathbb{Q}_{\ell}$-coefficients is nonvanishing only in degrees $0$ and $1$; a computation with the Lefschetz formula then yields 
\beq
|X_n(\F_q)| = q^n - q^{n-1} = \frac{q^n}{\zeta_{\A^1/\F_q}(2)}
\eeq
which is precisely analogous to the result in the number field case.

\item  A question with no obvious counterpart over a number field is: How does the number
of genus $g$ curves over $\F_q$ behave, as $g \rightarrow \infty$?

As already mentioned, Harer's theorem gives homological stability for $\MM_g$ (as orbifold) as $g \rightarrow \infty$.  But in this case, there is no bound of the form \eqref{moon}:  the Euler characteristic of $\MM_g$ grows superexponentially with $g$ and so, in particular, there is no bound on the Betti numbers in the unstable range analogous to Proposition~{\ref{Bettibound}.}  Thus, the homology stabilization does not enforce any regularity on $\frac{|\MM_g(\F_q)|}{q^{\dim \MM_g}}$, and it is not at all clear this ratio should be expected to approach a limit as $g \ra \infty$.     (See \cite{dejongkatz} for a discussion of this case, including the best known upper bounds for $|\MM_g(\F_q)|$.)

\item We expect the problem of counting points of bounded height on varieties over global fields to provide a very general example of the relation between stable homology and analytic number theory.   

Over $\F_q(t)$, this problem amounts to counting the number of $\F_q$-points on the space
of maps from $\P^1$ to an algebraic variety $X$; over $\C$, the homological stability
for such spaces is example 4 of \S \ref{context}. 

It has been observed by the first two authors \cite{LettertodeJong} that one can ``reverse'' the reasoning used in this paper, counting points over finite fields via the Hardy-Littlewood method and then applying the Lefschetz fixed point formula to obtain geometric information about $\mathrm{Hol}^d(\C\P^1, X)$.  In cases where the Hardy-Littlewood method does not apply, there is a notable similarity between the class of varieties $X$ such that $\mathrm{Maps}^d(S^2, X)$ is known to have stable homology, and those where the rational points of bounded height on $X/\Q$ is known to obey the asymptotic prediction of the Batyrev-Manin conjecture.

\end{enumerate}

\subsection{Acknowledgments}

The authors are very grateful for the assistance and advice offered by many mathematicians in the course of the present project:  among these we especially mention Jeff Achter, Manjul Bhargava, Henri Cohen, Ralph Cohen, Brian Conrad, Nathan Dunfield, Mike Fried, S{\o}ren Galatius, Ezra Getzler, Tom Graber, Richard Hain, Hendrik Lenstra, Martin Olsson, Ravi Vakil, Stefan Wewers, and Nathalie Wahl.  We are also grateful for the substantial efforts of several referees, whose comments substantially improved both the exposition and the argumentation of the paper. The first author was partially supported by NSF-CAREER Grant DMS-0448750 and a Sloan Research Fellowship; the second author was partially supported by a Packard fellowship, a Sloan Research Fellowship, and an NSF grant; the third author was partially supported by NSF grant DMS-0705428.

\subsection{Notation} \label{notn}

If a group $G$ acts on a topological space $X$, we will use the notation $X\modmod G$ for the Borel construction $X\modmod G := EG \times_G X$, where $EG$ is a contractible $G$-space with free $G$ action.

If $g$ is an element of a finite group $G$, we denote the order of $g$ by $|g|$.  If $g,h \in G$ we denote by $g^h$ the conjugate $h^{-1} g h$.

We will deal with graded modules $M = \bigoplus M_n$ over various graded rings. We will always understand the grading to be supported
on non-negative integers, i.e.  $M_n = 0$ for $n < 0$ for all our graded modules in this paper.

If $M = \oplus M_n$ is a graded module for a graded ring $R = \oplus R_n$, we write $\deg M$ to mean the maximal $n$ such that $M_n \neq 0$.  If there is no such $n$, we say $\deg M = \infty$.
  The notation $M[k]$ means ``$M$ shifted by $k$" -- in other words, $M[k]_n = M_{n-k}$ for all $n \geq k$.

If $r$ is a homogeneous element in $R_n$, we write $\deg(r) = n$.

\section{Definitions}
\label{se:definitions}

\subsection{Hurwitz spaces} \label{hur_sec}

We begin with a topological definition of Hurwitz spaces.  Their interpretation as moduli spaces of branched covers may not be immediate from this definition; we will return to that description afterwards.%

Let $D$ be a closed disc with a marked point $*$ on the boundary, and write $\Conf_n$ for the configuration space of $n$ unordered, distinct points in the interior of $D$.  Fix a basepoint $c_n = \{P_1, \ldots, P_n\}$ in $\Conf_n$, and recall also that the Artin braid group $B_n$ on $n$ strands is isomorphic both to the mapping class group of the punctured surface $\Sigma := D - \{P_1, \dots, P_n\}$,
and to the fundamental group of $\Conf_n$: 
$$ \pi_0 \Diff^+(\Sigma, \partial \Sigma)  \cong B_n \cong \pi_1(\Conf_n, c_n)$$
Referring to the presentation \eqref{artin-braid}, the first isomorphism identifies the braid $\sigma_j$ with the diffeomorphism that exchanges $P_j$ and $P_{j+1}$ by a half Dehn twist along a circle
 containing only these two punctures.
The second isomorphism carries $\sigma_j$ to a path in $\Conf_n$ that switches $P_j$ and $P_{j+1}$, while leaving immobile the other $P_i$.

For definiteness, we may take $D$ to be the closed disc centered at $0$ of radius $n+1$, take $P_j = j \in \mathbf{C}$, the diffeomorphism to be the half-Dehn twist around a a circle centered at $j+\frac{1}{2} \in \mathbf{C}$ of radius $3/4$, and the path switching $P_j$ and $P_{j+1}$ to be the path that rotates them both by a half-twist around a circle centered at $j+1/2$ and of radius $1/2$.

We recall that $\Conf_n$ is an Eilenberg-MacLane space\footnote{Indeed, consider the finite covering space of $\Conf_n$ that parameterizes $n$ {\em ordered} points; this space can be presented as an iterated fibration of punctured discs, and is thus aspherical.} $K(B_n, 1)$.
Let $\pi = \pi_1(\Sigma,  *)$.   Fix, for each $i$, an embedded loop $\gamma_i$ in $\Sigma$, based at $*$, and  winding once (counterclockwise) around the puncture $P_i$, and
not winding around any other puncture.   It is possible to do this in such a way that the $\gamma_i$ are not intersecting except at $*$. For definiteness, take $\gamma_i$ to be a straight path from $*$ to a point $P_i'$ very close to $P_i$, together with 
a small loop winding once around $P_i$ based at $P_i'$.

  The $\gamma_i$ freely generate $\pi$; thus we have specified an isomorphism between the free group $F_n$ on $n$ generators, and $\pi$.  Of course, different choices of $\gamma_i$ will yield differing isomorphisms.  We note, however, that any two choices of generators $\{\gamma_i\}$ and $\{\gamma_i'\}$ are related by a diffeomorphism\footnote{The map $h$ may be constructed by gluing together, for each $i$, diffeomorphisms from the punctured disk bounded by $\gamma_i$ to the one bounded by $\gamma_i'$, as well as the exterior of their unions.} $h: \Sigma \to \Sigma$ fixing $\partial \Sigma$.

Since $* \in \partial \Sigma$ is fixed by the action of the diffeomorphism group $\Diff^+(\Sigma, \partial \Sigma)$, there is a natural action of $B_n$ on $\pi$, and hence the set $\Hom(\pi, G)$ of homomorphisms from $\pi$ to any discrete group $G$.  Let $c \subseteq G$ be a conjugacy class or union thereof.  Write $\Hom^c(\pi, G)$ for the subset of homomorphisms $f: \pi \to G$ which carry each $\gamma_i$ into $c$.  

In principle, this subset is dependent upon our choice of generators $\gamma_i$.  Note, however, that any two loops $\gamma_i$, $\gamma_i'$ around $P_i$, of the type described above,  are conjugate. Thus we may reformulate the condition on $f$ without reference to the choice of $\{ \gamma_i \}$ to say that $f$ carries the free homotopy class of a loop around each puncture into $c$.  In this formulation, it is also apparent that this subset is invariant under the action of $B_n$, since any diffeomorphism of $\Sigma$ preserves the set of free homotopy classes of   loops around the punctures.

Finally, write $\Sur^c(\pi, G)$ for the subset of $\Hom^c(\pi, G)$ consisting of surjective homomorphisms (nonempty only if $c$ generates $G$),
and $\Sur(\pi, G)$ for the similarly defined subset of $\Hom(\pi, G)$.  These are evidently a $B_n$-invariant subset of $\Hom^c(\pi, G)$.

\begin{defn} \label{hur_defn}

Let $\tConf_n$ be the universal cover of $\Conf_n$, together with a fixed point $\tilde{c}_n$ above $c_n$; thus $B_n$ acts on $\tConf_n$ by deck transformations,
this action being uniquely specified by requiring the action of $b \in B_n$ on $\tilde{c}_n$ to coincide with the monodromy action of $B_n \simeq \pi_1(\Conf_n,c_n)$.   Define the \emph{Hurwitz spaces}
\begin{itemize}
\item $\Hur_{G, n} := \tConf_n \times_{B_n} \Hom(\pi, G)$, 
\item $\Hur_{G, n}^c := \tConf_n \times_{B_n} \Hom^c(\pi, G)$, and 
\item $\CHur_{G,n} = \tConf_n \times_{B_n} \Sur(\pi, G)$. 
\item $\CHur_{G, n}^c := \tConf_n \times_{B_n} \Sur^c(\pi, G)$.
\end{itemize}

\end{defn}

These are covering spaces of $\Conf_n$, finite sheeted when $G$ is finite.  Moreover, 
$$\left(  \mbox{fiber of $\Hur_{G,n} \rightarrow \Conf_n$ above $c_n$} \right) \cong \Hom(\pi, G) \cong G^n,$$
where the first map sends $  \tilde{c}_n  \times a \mapsto a \in \Hom(\pi, G)$, and the second map
uses the identification, described above, of $F_n$ with $\pi$. 
Similarly for the  other spaces: the fibers above $c_n$ are identified with $\Hom^c(\pi, G), \Sur(\pi, G)$ and $\Sur^c(\pi, G)$, respectively, which in turn are identified with  
the set of $n$-tuples of elements of $c \subseteq G$, the elements of $c^n$ that generate $G$, and the elements of $G^n$ that generate $G$ respectively. 

In all cases, with respect to these identifications,  the monodromy action of $\pi_1(\Conf_n, c_n) \simeq B_n$ on the fiber is identified with the braiding action  \eqref{braiding-action-def} of $B_n$ on $G^n$ (see, e.g., \cite{Artin}, equation (14)).

Note \label{Gactionpageref} also that $G$ acts on $\Hom(\pi, G)$ by conjugation in the target; the subsets $\Hom^c(\pi, G), \Sur(\pi, G),\Sur^c(\pi, G)$ are invariant under this action.  When given in terms of sets of $n$-tuples, the action is by termwise conjugation on the $n$-tuple.  Furthermore, this action commutes with the action of $B_n$ on the domain, and so yields an action of $G$ on all of the spaces  above.

Finally, we note that $\tConf_n$ is a contractible space with a free action of $B_n$, and so can regard $\Hur_{G, n}$ as the Borel construction $EB_n \times_{B_n} \Hom(\pi, G)$ (and similarly for the other spaces above).

\subsection{Interpretation as moduli spaces} \label{interps} 

A {\em marked $n$-branched $G$-cover of the disc} is a quintuple $(Y, p,\bullet, S, \alpha)$, where
\begin{itemize}
\item[-] $S \subset D$ is a set of $n$ distinct points in the interior of $D$; 
\item[-] $p: Y \ra D - S$ is a covering map; 

\item[-] $\alpha: G \ra \Aut(p)$ is a map inducing a simply transitive action of $G$ on each fiber;
\item[-] $\bullet$ is a point in the fiber of $p$ above $*$. 
\end{itemize}
Note that we do not restrict ourselves to {\em connected} covers $Y$.
We say two marked $n$-branched $G$-covers $Y$ and $Y'$ are isomorphic if there is a homeomorphism from $Y$ to $Y'$ over $D-S$, compatible with the remaining data.

Then we have bijections between the sets (a), (b), (c) described below:  
\begin{itemize}
\item[(a)]  points of $\Hur_{G,n}$ 
\item[(b)]   pairs $(S,f)$, 
 where $S \in \Conf_n$, and $f: \pi_1(D - S, *) \to G$ is a homomorphism
 \item[(c)] isomorphism classes of marked $n$-branched $G$-covers of $D$.  
 \end{itemize}
 
 For the bijection between (a) and (b),  regard
 elements of $\tConf_n$ above $S \in \Conf_n$ as a homotopy class of paths between $c_n$ and $S$, 
 in such a fashion that $\tilde{c}_n$ corresponds to the trivial path.  Such a path induces
 an isomorphism of $\pi=\pi_1(D-c_n, *)$ with $\pi_1(D-S,  *)$. Thus each point of $\tConf_n \times \Hom(\pi, G)$
 gives a pair $(S, \pi_1(D-S, *) \rightarrow G)$, and this descends to the desired bijection.
 
For the bijection between (b) and (c),  start with an $n$-branched $G$-cover, and identify the fiber above $*$ with $G$ through the map $g \mapsto g \bullet$; with respect to this identification, the action of $\pi_1(D-S,*)$ on $p^{-1}(*)$ is by right multiplication by $G$, and so defines a homomorphism $\pi_1(D-S,*) \rightarrow G$.

\begin{rem*} Hurwitz spaces appear in many places in the literature, and the definition admits many variants.  We emphasize that our Hurwitz spaces differ from many standard treatments in that we do not restrict our attention to connected $G$-covers (we reserve the notation $\CHur$ for such a Hurwitz space), and in that we select a marked point in the fiber over $*$.  This latter difference means that, by contrast with the notation in some of the literature, the points of our Hurwitz space with some fixed set of branch points $S \subset D$ are in bijection ``on the nose'' with the homomorphisms from the fundamental group of the punctured disc to $G$, not with the conjugacy classes of such homomorphisms.
\end{rem*}

The action of $G$ on $\Hur_{G,n}$ defined in the discussion following Definition~\ref{hur_defn} is given in these terms by moving the marked point, i.e., via the rule
\beq
 g(Y,p, \bullet, S, \alpha) = (Y,p, \alpha(g) \bullet, S, \alpha).
\eeq
Later we shall study the quotient of $\Hur_{G,n}$ by this $G$-action; this quotient space, which we denote $\Hur_{G,n}/G$, parametrizes $n$-branched $G$-covers without the specification of $\bullet$.

The subspace $\CHur_{G,n} \subseteq \Hur_{G,n}$ consists of the space of covers with full monodromy $G$ -- in other words, the covers corresponding to {\em surjective} homomorphisms $\pi_1(D-S, *) \ra G$.  The prepended ``C" is meant to recall that this space parametrizes {\em connected} $G$-covers of the disc.  The space $\CHur_{G,n}$ itself need not be connected in general (cf. \S \ref{introhur}).

\subsection{Combinatorial invariants.}
 $\Hur_{G,n}$ is usually disconnected, i.e. the action of $B_n$ on $\Hom(F_n,G)$ is typically not transitive.

We now describe some invariants of a cover $p$  which are constant on connected components of $\Hur_{G,n}$. 
By definition \ref{hur_defn}, it is equivalent to the combinatorial problem of specifying
a $B_n$-invariant function on $G^n$.

\begin{itemize}
\item The {\em global monodromy} of $p$ is the image of $\pi_1(D - S,*)$ in $G$.  In
combinatorial terms, this is the map $(g_1, \ldots, g_n) \rightarrow \langle g_1, \ldots, g_n\rangle,$
the subgroup of $G$ generated by the $g_i$. 
\item  The {\em boundary monodromy} of $p$ is the element of $G$ induced by a counterclockwise loop around $\partial D$. (More precisely, transport around such a loop
moves $\bullet$ to a point $g.\bullet$, for a unique $g \in G$.)

In combinatorial terms, this is the map $(g_1, \dots, g_n) \mapsto g_1 g_2 \ldots g_n$.
\item  For each $i$, the monodromy around a small loop encircling $P_i$ is an element of $G$, well-defined only up to conjugacy.  The resulting multiset of $n$ conjugacy classes of $G$ is called the {\em Nielsen class} of $p$. 

Combinatorially, the Nielsen class map associates to $(g_1, \ldots, g_n)$ the multiset obtained by replacing each $g_i$ with its conjugacy class.
\end{itemize}

Fixing the global monodromy, boundary monodromy, and Nielsen class of a cover specifies a subspace of $\Hur_{G,n}$; although it may be disconnected, there are no ``obvious'' invariants further separating connected components. 

For a conjugacy class $c \subseteq G$ we note that $\Hur_{G,n}^c$ is the subspace of $\Hur_{G,n}$ consisting of covers whose Nielsen class is $n$ copies of $c$.
Our main goal in the present paper is to study the homology groups $H_p(\Hur_{G,n}^c)$, especially their asymptotic behavior as $n$ grows with $G$ and $c$ held fixed.  It is also natural to consider the larger spaces where the monodromy is drawn not from a single conjugacy class $c$ but from a fixed union of conjugacy classes, or for that matter from the whole group.  We do not pursue this generalization in the present paper.

\begin{prop} \label{Bettibound}
$\Hur_{G,n}$ and $\Hur_{G,n}^c$ are both homotopy equivalent to CW complexes with at most $(2 |G|)^n$ cells. 
\end{prop}
\proof 
Since $\Hur_{G,n}$ and $\Hur_{G,n}^c$ are
both coverings of $\Conf_n$ with fibers of size $\leq |G|^n$,
 it suffices to check that $\Conf_n$  is homotopy equivalent
 to a CW complex with $\leq 2^n$ cells.
 For this see, e.g., \cite{Salvetti} or \cite{Charney-Davis}.   
\qed 

 Proposition~\ref{Bettibound} is critical for the arithmetic applications; in the Lefschetz trace formula it is this fact that allows us to neglect the contribution of cohomology classes in the unstable range.

\subsection{Gluing maps} \label{gluing} 

Arising from the natural inclusions $B_n \times B_m \rightarrow B_{n+m}$
and $G^n \times G^m \rightarrow G^{n+m}$, we obtain a map on Borel constructions 
$$(EB_n \times_{B_n} G^n) \times (EB_m \times_{B_m} G^m) \to EB_{n+m} \times_{B_{n+m}} G^{n+m}$$
which defines (up to homotopy) a gluing map: 
$$\Hur_{G,n} \times \Hur_{G,m} \longrightarrow \Hur_{G,n+m}.$$
This multiplicative structure is associative up to homotopy since both maps inducing it are; collectively, they make the union of the Hurwitz spaces into an $H$-space.  

Geometrically, these maps associate to a pair of branched covers $Y_1$, $Y_2$ of $D$ a new cover, $Y_3$.  Pick two standard disjoint, embedded loops $\gamma_1$, $\gamma_2$ in $D$ based at $*$; then the restriction of $Y_3$ to the interior of the region described by each $\gamma_i$ is isomorphic to $Y_i$.  On the complement of the loops, $Y_3$ is the trivial $G$-bundle extension.

Similarly, we have a multiplication
$$\Hur_{G,n}^{c} \times \Hur_{G,m}^{c} \longrightarrow \Hur_{G, n+m}^{c}.$$

We note that these gluing maps are equivariant for the action of $G$ (where $G$ acts on all three factors
in the fashion defined after Definition \ref{hur_defn}) when we take the model given by the Borel construction.

\section{The ring $R$ of connected components}  \label{Rdef}
 Let $k$ be a field of characteristic prime to $|G|$. 
Then the graded ring $$ R = \sum_n H_0(\Hur_{G,n}^c,k),$$
 inherits, from the multiplication  on Hurwitz spaces (\S \ref{gluing}), 
the structure of a non-commutative $k$-algebra; moreover, the higher homology of Hurwitz
spaces carries the structure of $R$-module.

\begin{defn} We say the pair $(G,c)$ has the {\em non-splitting property} if $c$ generates $G$ and, moreover, for every subgroup $H$ of $G$, the intersection $c \cap H$ is either empty or a conjugacy class of $H$.\end{defn}

Our main result in this section is Lemma~\ref{le:rurfinite}, which implies that {\em if $(G,c)$ has the non-splitting property, 
then there exists a central homogeneous element $U \in R$ so that the degree 
of $R/UR$ is finite. }

Before discussing $R$, we begin by giving the basic example of non-splitting pairs:

\begin{lem}  Let $G$ be a finite group whose order is $2s$, for $s$ odd.   Then there is a unique conjugacy class of involutions $c \subset G$; if $c$ generates $G$, then $(G,c)$ is non-splitting. \label{le:dihedralnon-splitting}
 \end{lem}

\begin{proof}  
The fact that all involutions are conjugate follows from conjugacy of $2$-Sylow subgroups; %
any subgroup $H$ of $G$ containing an involution has order $2s'$ for $s'$ odd, and 
 the non-splitting follows from the uniqueness assertion applied to $H$. 
\end{proof}

A group $G$ as in the Lemma is necessarily isomorphic to $G_0 \rtimes (\Z/2 \Z)$
for  some group $G_0$ of odd order.  In fact, these are the only cases of non-splitting pairs where $c$ is a involution.\footnote{This fact follows from Glauberman's $Z^*$ theorem, as Richard Lyons explained to us; the authors thank \texttt{mathoverflow.net} for providing a forum where we could ask about this and be provided with an authoritative reference.}  There are other non-splitting pairs:  for example, $G = A_4$ and $c$ one of the classes of $3$-cycles.

{\em For the remainder of this paper all theorems have as a hypothesis that $(G,c)$ has the non-splitting property.}

\subsection{Combinatorial description of $R$.}

\label{ss:combinatoricsr}

The graded ring $R$ has a very concrete description: Let $\tilde{S}$ be the set of tuples of elements from $c$ (of any nonnegative length), and let $S$ be the quotient of $\tilde{S}$ by the action of the braid group.  Then $S$ is a semigroup under the operation of concatenation, and $R$ is the semigroup algebra $k[S]$.   
We let $S_n = c^n / B_n$ be the subset of $S$ consisting of elements of degree $n$; 
for $s \in S$ (considered as an element of $R$) write $\partial s \in G$ for the boundary monodromy of $s$; if $s$ is represented by $(g_1, \dots g_n)$, then $\partial s = g_1 \cdots g_n$.

  $R$ is generated over $k$ by degree $1$ elements $\set{r_g}_{g \in c}$, subject to the relations
\begin{equation}
r_g r_h = r_{g h g^{-1}} r_g
\label{eq:relation}
\end{equation}

We occasionally denote $r_g$ by $r(g)$ if the group element in question is too typographically complicated to fit in a subscript. We note that we learned the idea of using the semigroup $S$ to study connected components of Hurwitz  spaces from the Appendix to \cite{friedvolklein}.

\begin{prop}  Let $g \in c$.   For sufficiently large $n$, every $n$-tuple $(g_1, \ldots, g_n)$ in $\tilde{S}$ whose elements generate $G$ is equivalent under the braid group to an $n$-tuple $(g, g'_2, \ldots, g'_n)$, where $g'_2, \ldots, g'_n$ generate $G$.
\label{pr:cp1}
\end{prop}

This in particular implies stability for the zeroth Betti number: $b_0(\CHur_{G,n}^c)$ is independent of $n$, for sufficiently large $n$.  Regarding $\pi_0(\CHur_{G, n}^c)$ as a subset of $S_n$, this shows that the map $\pi_0(\CHur_{G, n}^c) \to \pi_0(\CHur_{G, n+1}^c)$ given by adding $g$ at the beginning of an $n$-tuple is surjective for $n$ sufficiently large.  Both sets are finite, so this is eventually a bijection.

\begin{proof}
This is well-known (see e.g. \cite{friedvolklein}) but for completeness we include a proof here.  It is clear that, by repeated action of the braid action \ref{braiding-action-def}, we can pull the $i$th monodromy element to the beginning of the $n$-tuple; that is, the $n$-tuple $(g_1, \ldots, g_n)$ is equivalent to $(g_i, g'_1, \ldots, g'_{n-1})$ for some  $(n-1)$-tuple $(g'_1, \ldots, g'_{n-1})$.  Write $d$ for the order of an element of $c$.   If $n > d|c|$, some element $g'$ of $c$ occurs at least $d+1$ times in $(g_1, \ldots, g_n)$; we can use the braid action to pull these back to the front, forming an $n$-tuple
\beq
\underline{g} = (g',g', \ldots, g', g'_1, \ldots, g'_{n-d-1})
\eeq
equivalent under the braid action to $(g_1,\ldots, g_n)$.  The elements of this $n$-tuple generate $G$, whence $g',g'_1, \ldots, g'_{n-d-1}$ generate $G$.

Now the fact that $(g')^d = 1$ implies that, for any $k$-tuple $h_1, \ldots, h_k$, the $d+k$-tuple 
\beq
(\underbrace{g',g',\ldots,g'}_{d},h_1, \ldots, h_k)
\eeq
is equivalent under braiding to
\beq
(hg'h^{-1}, \ldots, hg'h^{-1}, h_1, \ldots, h_k)
\eeq
where $h$ is the product $h_1 \ldots h_k$; this equivalence is implemented by the braid that winds the first $d$ strands around the last $k$ strands -- i.e.,
we first braid the $g'$s to the right to obtain $(h_1, \ldots, h_k, g' , \dots,   g' )$, and then braid the $h$s to the right to obtain the tuple above.   Since the braid action can conjugate the $d$ copies of $g'$ by $h_1 \ldots h_k$ and by $h_1 \ldots h_{k-1}$, it can also conjugate those $d$ copies of $g'$ by $h_k$ alone.  Repeating this process, we find that
\beq
(g',g',\ldots,g',h_1, \ldots, h_k)
\eeq
is equivalent to 
\beq
(g'',g'', \ldots, g'', h_1, \ldots, h_k)
\eeq
for any $g''$ which is conjugate to $g'$ via an element in the group generated by $h_1, \ldots, h_k$.

Now the final $n-d$ elements of $\underline{g}$ generate $G$, as we have seen.  Since $g'$ is conjugate to $g$ via some element of $G$, the argument above shows that $\underline{g}$ is equivalent under the braid action to
\beq
(g,g,\ldots, g,g', g'_1, \ldots, g'_{n-d-1})
\eeq
This proves the proposition. 
\end{proof}
 
Recall that if $M = \oplus M_n$ is a graded $R$-module, we write $\deg M$ to mean the maximal $n$ such that $M_n \neq 0$.  If there is no such $n$, we say $\deg M = \infty$.  In the following Lemma, we prove a finiteness condition on $R$ which will turn out to imply all the homological properties of the category of $R$-modules that we require for the proof of the main theorem.

\begin{lem}  Suppose that $(G,c)$ has the non-splitting property. 
For an integer $D$, write
\beq
U_D = \sum_{g \in c} r_g^{D|g|},
\eeq
so that $U_D$ is in the center of $R$. 
Then there exists a $D$ such that the degree of both kernel and cokernel
of $$R \stackrel{U_D}{\longrightarrow} R, \ \ r \mapsto U_D r$$ 
are finite.  %
\label{le:rurfinite}
\end{lem}

\begin{proof}

The main ingredient is Proposition~\ref{pr:cp1}.  The non-splitting property for $G$ is used in an essential way, allowing the application of this Proposition to the pair $(H, c \cap H)$ for subgroups $H \leq G$.

Within the present proof, we refer to the subset of $S_n$ consisting of braid orbits on $n$-tuples generating $H$ as $S_n(H)$.
We first  show that for every subgroup $H$ of $G$, every element $g$ of $c \cap H$, and every sufficiently large $n$, the map \begin{equation} \label{bij} S_n(H) \rightarrow S_{n+|g|}(H), \ \ 
 s \mapsto r_g^{|g|} s \end{equation} 
 is bijective. 
It suffices to show that this map is surjective, for large enough $n$: since all the sets
involved are finite, it must then be eventually bijective, which is the assertion to be proved.

If $c \cap H$ is empty, then so too is $S_{N}(H)$; the claim is vacuously true.  If not, take $s \in S_n(H)$; for sufficiently large $n$, Proposition~\ref{pr:cp1} (applied to the group $H$ and conjugacy class $c \cap H$)
 shows that $$s  \sim r_g s', \ \ s' \in S_{n-1}(H).$$ 
 Note it is exactly at this point that the  non-splitting property enters our argument: we used the fact that $c \cap H$ is a single conjugacy class. See also the discussion after the Theorem in \S \ref{introhur}.
  
Increasing $n$ as necessary, we can repeat this process $|g|$ times; this
shows that $r_g^{|g|}$ induces a surjective map, as desired. 
 
It is not clear that different choices of $g \in c \cap H$ induce the {\em same} bijection in \eqref{bij}. 
For $g_1, g_2 \in c \cap H$, observe that multiplication by $ A := r_{g_1}^{|g_1|}$ and $ B := r_{g_2}^{|g_2|}$
commute (as self-maps of $S$).   For sufficiently large $n$,  the map  $B$ is a bijection from $S_n(H)$ to $S_{n+|g|}(H)$ 
and so has an inverse, which we denote $B^{-1}: S_{n + |g|}(H) \ra S_n(H)$.  So, again for sufficiently large $n$, 
$  A \circ B^{-1}$  is a permutation of $S_n(H)$.  
Let $D$ be chosen so
that every permutation of every $S_n(H)$ has order dividing $D$. 
Then $(A \circ B^{-1})^D = A^D \circ B^{-D} $ induces the identity map on $S_n(H)$ for large enough $n$. 
Thus, for such $D$,  the map:
$$ S_n(H) \stackrel{r_g^{D|g|}}{\longrightarrow} S_{n+D|g|}(H)$$
is -- for large $n$ -- a bijection and is independent of $g \in c \cap H$.

For $m \geq 1$,  set $F_m R$ to be the subspace of $R$ generated
 by elements in $S_n(H)$,  as $n$ ranges over nonnegative integers and $H$ ranges over subgroups of order at least $m$.  Note that $U_D$ preserves
 $F_m R$; we now show, by descending induction on $m$, that
  \begin{equation} \label{fh} U_D : F_m R_n \rightarrow F_m R_{n + D |g|} \end{equation} 
 is an isomorphism
 for sufficiently large $n$. Again, it is sufficient to show that \eqref{fh} is surjective for sufficiently large $n$. 

The inductive claim is valid for $m> |G|$ trivially. 
 Now suppose it is true whenever $m > m_0$.

 Let $H$ be a subgroup of $G$ of size $m_0$. 
It suffices to check that -- for large enough $N$ -- every $y \in S_{N}(H)$ belongs to the image of $U_D$.  Again, this is vacuously true if $c \cap H = \emptyset$.  Otherwise, by what
 we have shown, there exists $x \in S_{N - D|g|}(H)$ such that $r_g^{D |g|} x = y$
 for all $g \in c \cap H$. Therefore, 
 $$U_D x = |c \cap H| y + y', \mbox{ where $y' = \sum_{g \in c \setminus (c \cap H)} r_g^{D|g|} x \in F_{m_0+1} R_N$}.$$ 
By inductive assumption, there exists, for sufficiently large $N$, $x' \in F_{m_0+1}R_{N-D|g|}$ such that $U_D x' = y'$. Consequently, $U_D(x-x') = |c \cap H| y$. Since $c \cap H$ is a conjugacy class of $H < G$, the size of $c \cap H$ divides $|G|$
and is therefore invertible in $k$, we conclude that $y$ is in the image of $U_D$, as desired. 
 \end{proof}

We note that the assumption that $|G|$ is invertible in $k$ was used in a substantial way in this proof.  Consequently, we do not expect that rational homological stability for Hurwitz spaces can be improved to an integral result, since Theorem \ref{th:stability} depends in a basic way upon this fact.  However, the proof of Lemma \ref{le:rurfinite} suggests that an integral result may be possible for subspaces of Hurwitz spaces (such as $\CHur_{G, n}^c$). 

\section{The $\mathcal{K}$-complex associated to an $R$-module} \label{Rhomology}

This section is solely concerned with homological properties of the ring $R$ introduced in \S \ref{Rdef}. 
In particular, we associate to each $R$-module $M$ a certain Koszul-like complex,
the $\mathcal{K}$-complex (\S \ref{kcomplex}).
 We shall see in \S \ref{arc-complex} that the homology of Hurwitz spaces can be inductively expressed
in terms of $\mathcal{K}$-complexes formed from homology of {\em smaller} Hurwitz spaces.

Our main result, Theorem \ref{Koszulbound}, is that the higher homology of the $\mathcal{K}$-complex is controlled by its $H_0$ and $H_1$. The overall thrust of this section can be roughly summarized by the slogan
 ``$R$ behaves as if it had cohomological dimension $1$.'' 

Let us explain why this point of view is useful for proving homological stability for Hurwitz spaces.  In most situations where homological stability is understood, one has a sequence of (usually connected) spaces $X_n$ and stabilization maps $f_n: X_n \to X_{n+1}$; the goal is to show that each $f_n$ induces homology isomorphisms in a range of dimensions.  Let $X= \sqcup_n X_n$, and consider the homology
$$M_p = H_p(X) =\oplus_n H_p(X_n)$$
Give $M_p$ the structure of a $k[x]$-module by making the indeterminate $x$ act via the stabilization map.  $M_p$ admits a grading by the number $n$, and $x$ acts as a degree 1 operator.  Homological stability is rephrased as the statement that $x$ is an isomorphism in sufficiently high degree.  Equivalently, we need the quotient and $x$-torsion
$$\begin{array}{ccc}
\Tor_0^{k[x]}(k, M_p) = M_p/xM_p & {\rm and} & \Tor_1^{k[x]}(k, M_p) = M_p[x]
\end{array}$$
to be concentrated in low degrees.

Approaching homological stability for Hurwitz spaces this way, one immediately runs into a problem: there are many natural stabilization maps, one for each isomorphism class of branched cover of the disk.  This more complicated structure is encoded, however, in the ring $R$ of connected components of the Hurwitz space, which replaces $k[x] = H_0(X)$ above.  As we have seen in Lemma \ref{le:rurfinite}, $R$ itself satisfies a form of stability so long as $(G,c)$ is nonsplitting.  This fact, combined with control of the homological algebra of $R$ (developed in this section) ultimately gives rise to homological stability for the Hurwitz spaces.

Throughout this section $(G,c)$ is non-splitting, and we take $U$ to be the central element $U_D$ defined in 
Lemma \ref{le:rurfinite}. 

\subsection{} \label{kcomplex}
Let $M$ be any graded left $R$-module.  
We may define a ``Koszul-like'' complex  ($\mathcal{K}$-complex for short) associated to $M$,
where\footnote{Recall (\S \ref{notn}) that $[i]$ denotes a shift in grading by $i$.} $\Koszul(M)_q = k[c^q] \otimes_{k} M[q]$, $q \geq 0$:  that is, we have 
\begin{equation} \label{koszuldef} 
\Koszul(M) := \ldots \ra k[c^q] \otimes_{k} M[q] \rightarrow k[c^{q-1}] \otimes_{k} M[q-1] \rightarrow \dots \rightarrow k[c] \otimes_{k} M[1]\rightarrow M[0] \end{equation} 
where  the differential $\Koszul(M)_{q+1} \rightarrow \Koszul(M)_{q}$ is described by:
$$(g_0, \dots, g_q) \otimes m \mapsto \sum_{i=0}^q (-1)^i (g_0, \dots, \widehat{g_i}, \dots, g_q) \otimes 
r(g_i^{g_{i+1} \cdots g_q})   m.$$

 We equip $k[c^q]$ with the trivial grading, i.e. the one concentrated in degree $0$. 
Then  the differentials preserve the grading. Moreover, if $M = R$, 
each homology group of $\Koszul$ is equipped with the natural structure of a graded {\em right} $R$-module. 
This notion is chosen to model  a complex which will arise  in our  study of the arc complex (\S  \ref{arc-complex}).

\begin{thm} \label{Koszulbound}
Suppose $(G,c)$ is non-splitting; let $M$ be a graded left $R$-module,  and let $h_i = \deg(H_i(\Koszul(M)))$. 
Then there exists a constant $A_0 = A_0(G,c)$ so that 
\begin{equation}  \label{theorem-equation} h_q \leq  \max(h_0, h_1) + A_0 q \ \ \ (q> 1).  \end{equation} 
 Moreover $M \stackrel{U}{\rightarrow} M$ is an isomorphism in source degree $\geq \max(h_0, h_1) + A_0$
 (that is to say, the induced map $M_i \rightarrow M_{i+\deg(U)}$ is an isomorphism for $i \geq  \max(h_0, h_1) + A_0$).

 Finally, in the case $M = R$, $h_0$ and $h_1$ are both finite. 
 \end{thm}

 The explicit value of $A_0$
 that comes from the proof is given in \eqref{constants}; we have not attempted to optimize
 this value as far as possible, since the precise bound makes no difference to our end goal.
 
This Theorem is fundamental to the proof of our main result.
A basic tool in its proof is comparing the homological algebra of $R$ and
the commutative, central subring $k[U]$.  The {\em centrality} of $U$ is vital to our argument, and will be used without comment repeatedly.

In what follows, we use the following notations and conventions:
\begin{itemize}
\item We denote the two-sided ideal $\oplus_{n > 0} R_n$ by $R_{>0}$; we give the field $k$ the structure of $R$-bimodule by identifying it with $R/R_{>0}$.   
\item  For $M$ a graded left $R$-module, we denote by $H_i(M)$ the graded left $R$-module $\Tor_i^R(k,M)$.  In particular, $H_0(M) = M/R_{>0}M$. 
\item We set $\Rbar = R/UR$; it is an $R$-bimodule.
\item    For an $R$-module $M$, write $M[U]$ for the $U$-torsion in $M$, 
  i.e. the kernel of the ``multiplication by $U$'' map $M \stackrel{U}{\rightarrow} M$.     
  \item Finally, recall (\S \ref{notn}) that if $M$
is a graded  $R$-module, we write $\deg(M)$ for the largest degree $n$ such that $M_n \neq 0$, if it exists; otherwise $\deg(M)=\infty$. 
 \end{itemize}

Thus, with these conventions, $\deg(\Rbar)$ and $\deg(R[U])$ are both finite, because of Lemma \ref{le:rurfinite}.  

\subsection{Comparing the homological algebra of $R$ and $k[U]$.}
Consider the functors
$$ \mbox{Left graded $R$-modules} \stackrel{f}{\longrightarrow} \mbox{Left graded $\bar{R}$-modules} \stackrel{g}{\longrightarrow}
\mbox{graded $k$-vector spaces},$$
where $f$   sends $M$ to $\bar{R} \otimes_{R} M$, and  $g$   sends $\bar{M}$ to $k \tensor_{\Rbar} \bar{M}$. 
Both $f$ and $g$ are right exact, and admit left derived functors.  Since $f$ carries free left $R$-modules
  to free left $\Rbar$-modules, and each left $R$-module has a resolution by free left $R$-modules, 
  we have a spectral sequence  \begin{equation} \label{groth}
\Tor_i^{\bar{R}}(k,   \Tor_j^R(\bar{R}, M)) \Rightarrow \Tor^R_{i+j}(k, M). \end{equation}

\begin{lem}  \label{easy}    Let $M$ be a graded left $R$-module and $N$ a graded right $R$-module. Then  
\begin{equation} \label{degbound1}
\deg(N \tensor_R M) \leq \deg(N)+  \deg(H_0(M)) 
\end{equation} 
\end{lem}

\begin{proof}

When $\deg(N) = \infty$ or $\deg(H_0(M)) = \infty$, the assertion is vacuously true, so we assume both numbers are finite from now on. 
 We will use similar reasoning in the proofs that follow, without explicit mention.
 
 The case of $N=k$ in degree $0$ follows at once because 
  $H_0(M) = k \otimes_R M$.

The general case reduces to this: 
Consider an exact sequence of graded right $R$-modules $0 \ra N_1 \ra N_2 \ra N_3 \rightarrow 0$; if the assertion holds
for $N = N_1, N_3$ it holds also for $N=N_2$.   The assertions are also unchanged by applying degree shifts to $N$.   
Now we proceed, by induction, on the largest degree $a$ in which $N_a \neq 0$. 
Note that $N_a$ is automatically 
an $R$-submodule, isomorphic as $R$-module to a sum of copies of $k$, and so the 
assertion is known for $N_a$. This gives, in particular,  the case $a=0$ of the induction; and for  $a>0$ we use the exact sequence $N_a \rightarrow N \rightarrow N/N_a$
and induction. 
\end{proof}

\begin{lem}  Let $\bar{M}$ be a left graded $\Rbar$-module.  Then the degree of $ \Tor^{\Rbar}_i(k,\bar{M})$ is at most $ (\deg \Rbar) i+ \deg(\bar{M}) $. 

\label{le:rbarbound}
\end{lem}

\begin{proof}
Let
\beq
\cdots \ra P_2 \ra P_1 \ra P_0 \ra k
\eeq
be a resolution of $k$ by projective right graded $\Rbar$-modules. 
 We note that $P_i$ can be chosen to be generated in degree at most $i \deg(\Rbar)$. Indeed, this is so for $i=0$, and we construct $P_i$ by taking the free module 
 on a set of generators for $\mathrm{ker}(P_{i-1} \rightarrow P_{i-2})$; 
 by inductive hypothesis $P_{i-1}$ is  supported in degree at most $\deg(\Rbar) i$. 
  Using Lemma \ref{easy} and the fact that every $\Rbar$-module is an $R$-module, we have 
\beq
\deg(\Tor^{\Rbar}_i(k,\bar{M}) )\leq \deg(P_i \otimes_{\Rbar} \bar{M}) \leq \deg(H_0(P_i)) + \deg(\bar{M})
\leq
(\deg \Rbar)i+ \deg(\bar{M}) .
\eeq
\end{proof}

In the lemmas that follow, we shall adopt the following notation for $M$ a graded left $R$-module: 
\begin{equation} \label{Adef}  A(M) = \max(\deg M[U], \deg M/UM), \end{equation} 
\begin{equation} \label{deltaMdefinition}  \delta(M)  = \max \left( \deg  \Tor_0^R (\bar{R}, M), \deg \Tor^R_1(\bar{R}, M)\right) \end{equation}
In both cases, we allow the value $\infty$ if the degrees in question are not finite. 
Note that $A(R) =\max(\deg R[U], \deg R/UR)$ is finite  by virtue of Lemma \ref{le:rurfinite}.

We will also use several constants in the proofs that follow. We summarize them here for convenience:
\begin{eqnarray} \label{constants} A_1 &=&\Aone, \\ \nonumber
 A_2 &=& \Atwo ,  \\ \nonumber 
 A_0 & =&  5A_1 + A_2 = 6 A(R) + \deg(U). \end{eqnarray}

\begin{lem} \label{opm2lem}  
Let $M$  be a graded left $R$-module.  
Then 
 \begin{equation}
 \label{opm2}  A(M) \leq \delta(M) + A_1,\end{equation}
where $A_1=  \Aone$ as in \eqref{constants}.
\end{lem}
\begin{proof}  

    The bound on the degree of $M/UM \simeq  \bar{R} \otimes_R M = \Tor^R_0(\bar{R}, M)$
is clear by definition of $\delta(M)$.  It remains to bound $\deg M[U]$.

In what follows, $UR$ denotes the two-sided ideal of $R$ generated by $U$. Let us write
$N = (UR) \otimes_R  M $.  Then $N$ has the structure of a graded left $R$-module. 
Now we may regard multiplication by $U$ as a map
  $M \stackrel{U}{\rightarrow} M $ of degree $\deg(U)$ which can be factored as  
$$ M \stackrel{\alpha}{\rightarrow}  N \stackrel{\beta}{\rightarrow}  M$$
where $\alpha$ is the map $M = R \otimes_R M \rightarrow UR \otimes_R M$ 
given by $s \otimes m \mapsto  Us \otimes m$, 
and  $\beta(s \otimes m) = sm$; 
thus $\alpha$ is of degree  $\deg(U)$ and $\beta$ is of degree $0$. 
Consequently, we get an exact sequence (i.e., exact at the $\ker(\alpha)$ and $M[U]$ terms):
\begin{equation} \label{froo} 0 \rightarrow \ker(\alpha) \rightarrow  M[U] \stackrel{\alpha}{\rightarrow} \ker(\beta) \end{equation}  
By tensoring the short exact sequences $UR \hookrightarrow R \twoheadrightarrow \bar{R}$  
and   $R[U] \hookrightarrow R \twoheadrightarrow UR$ with $M$, we obtain
$$ \ker(\beta) \simeq \Tor^R_1(\bar{R}, M)  \mbox{ and }   \ker(\alpha) \twoheadleftarrow M  \otimes_{R} R[U].$$
This turns \eqref{froo} into the following sequence, exact at the middle term
\begin{equation} \label{froo2}    R[U] \otimes_R M \rightarrow M[U] \rightarrow \Tor^R_1(\bar{R}, M). \end{equation} 
 where the first map is of degree $0$ and the second map is of degree $\deg(U)$. 
 We also observe that, by the definition  \eqref{deltaMdefinition} of $\delta(M)$, we have $\deg \Tor^R_1(\bar{R}, M) \leq \delta(M)$.

Now
 \beq
\deg(R[U] \otimes_R M) \leq \deg(H_0(M)) + \deg(R[U])
\eeq
by Lemma \ref{easy}.  We note that $\deg H_0(M) \leq \deg M/UM \leq \delta(M)$.  Thus we arrive at the bound: 
\begin{eqnarray*} \deg(M[U]) & \leq& \delta(M) + \deg(R[U])  \\ &\leq& \delta(M) + A(R). \end{eqnarray*}

 \end{proof}

\begin{lem}  \label{opm3lem}
Let $M$ be a graded left $R$-module; then, with notation as above, 
\begin{equation} \label{opm3} \deg \Tor_i^R(\bar{R}, M) \leq \delta(M) + A_1 i \end{equation} 
where $A_1=  \Aone$ as before. 
\end{lem}
\begin{proof}
(We thank the referee for the proof that follows, which greatly improves on 
the prior version.) 
Firstly, the assertion is clear for $i = 0$
and for $i =1$ by definition 
of $\delta(M)$. 
 We proceed now by induction on $i$. 
 
Note that $R[U]$ is in fact a right $\bar{R}$-module.
 Construct a resolution of $R[U]$ by free right $\bar{R}$-modules:
\begin{equation} \label{refseq}  \cdots \rightarrow Q_2 \rightarrow Q_1 \rightarrow Q_0 \twoheadrightarrow R[U]. \end{equation} 
 Here, we may suppose that $Q_i$ is generated as a $\bar{R}$-module by elements of degree $\leq \deg(R[U]) + i \deg{\bar{R}} \leq (i+1)A_1$
 by the argument of Lemma \ref{le:rbarbound}.
  Combine  \eqref{refseq}  with $R[U] \rightarrow R \stackrel{U}{\rightarrow} R \rightarrow \bar{R}$ to obtain a resolution of $\bar{R}$
  by right $R$-modules
 $$\cdots \rightarrow Q_2 \rightarrow Q_1 \rightarrow Q_0 \rightarrow R \stackrel{U}{\rightarrow} R \to \bar{R} \to 0$$
Call this complex $P_{\bullet}$,
 so that $P_1=P_0=R$ and $P_i = Q_{i-2}$ when $i \geq 2$.   We now have a hyperhomology sequence
 $$E^1_{ij} = \Tor^R_i(P_j, M) \implies \Tor^R_{i+j}(\bar{R}, M).$$
 
That shows that, for $b \geq 2$,  $\Tor^R_{b}(\bar{R}, M)$ admits a filtration, whose associated graded is a subquotient of
\begin{equation} \label{refproof} \bigoplus_{u=0}^{b-2} \Tor^R_u(Q_{b-u-2}, M). \end{equation}
By construction, $Q_{b-u-2}$ is a free $\bar{R}$-module which was generated
in degree at most $(b-u-1) \cdot A_1$, so the degree of \eqref{refproof} -- and so also of $\Tor^R_b(\bar{R}, M)$ -- 
is at most the maximum over $0 \leq u \leq b-2$ of   \begin{eqnarray}  && (b-u-1) A_1 + \deg \Tor^R_u(\bar{R}, M) 
 \\   &\leq_{(i)}& (b-u-1) A_1 + \delta(M)  +  A_1 u 
 \\ & \leq_{ \  \  \ } & b A_1 + \delta(M),    \end{eqnarray}
 where step (i) follows from the induction hypothesis, since $u < b$.
  \end{proof}

\begin{lem} \label{NAK} 
Let $\bar{M}$ be any left $\bar{R}$-module; then $ \deg(\bar{M})
\leq \deg(k \tensor_{\Rbar} \bar{M} )  + \deg(\Rbar)$. 
\end{lem}
\begin{proof}
This  follows from a form of  Nakayama's Lemma: Choose homogeneous elements  $x_1, \dots $ of $\bar{M}$ projecting to a $k$-basis
for $k \tensor_R \bar{M}$; the quotient $\mathcal{Q} = \bar{M}/ \sum \Rbar x_i$
is a graded $\bar{R}$-module and satisfies $k\tensor_{\Rbar} \mathcal{Q }= 0$. We claim $\mathcal{Q}$ is zero; if not
let $j$ be the smallest integer such that $j$th graded piece of $\mathcal{Q}$ is nonzero. 
  The image of this graded piece in $k \tensor_{\Rbar} \mathcal{Q} $ cannot be trivial, since  this tensor product
  is the same as the quotient of $\mathcal{Q}$ by $\mathcal{Q}' := \ker(\Rbar \rightarrow k) \cdot \mathcal{Q}$, 
  and $\mathcal{Q}'$ is supported in degrees strictly greater than $j$. 
\end{proof}

\begin{lem} \label{DLbound}
 For any graded  left $R$-module $M$, we have  $$ \delta(M) \leq  \max(\deg(H_0(M)),\deg(H_1(M))) +4A_1.$$ \end{lem}

\begin{proof} 
Now 
\begin{equation} \label{firstbound} \deg  (  \Rbar \otimes_R M) \leq  \deg H_0(M) + \deg(\bar{R}), \end{equation} 
by Lemma \ref{easy}. 

We must now bound  $\deg \Tor^R_1\left(\bar{R}, M\right)$. 
The spectral sequence \eqref{groth} gives an exact sequence
$$\Tor_2^{\Rbar}(k,  \Rbar \otimes_R M) \ra k\tensor_{\Rbar}   \Tor^R_1(\Rbar, M) \ra H_1(M),$$ 
and  applying  Lemma \ref{NAK} to the left $\bar{R}$-module $\Tor^R_1(\Rbar, M)$, we see that
\begin{equation} \label{frugga} \deg \Tor^R_1(\Rbar, M) \leq  \deg \Rbar + \max(\deg H_1(M), \deg \Tor_2^{\Rbar}(k,  \Rbar \otimes_R M)).\end{equation}

By Lemma \ref{le:rbarbound} we have
\begin{eqnarray}  \deg \Tor_2^{\Rbar}(k,  \Rbar \otimes_R M)  & \leq &  2 \deg(\bar{R}) + \deg \Rbar \otimes_R M  \\ 
& \leq & 
 3 \deg(\bar{R}) + \deg H_0(M) \end{eqnarray}
 where we used \eqref{firstbound} again for the last inequality.  Combining this with \eqref{frugga}  and the fact that $\deg(\bar{R}) \leq A_1$ yields the desired bound on $\deg \Tor^R_1\left(\bar{R}, M)\right)$. 
 \end{proof}

\begin{prop}     Let $M$ be a graded left $R$-module and $N$ a graded right $R$-module. Then 
 \beq \label{sec}
\deg(\Tor^R_i(N,M)) \leq \deg(N) + \max(\deg(H_0(M)),\deg(H_1(M)))  + A_1 i +  4A_1. 
\eeq
for all $i \geq 0$.
\label{pr:cd1}
\end{prop}

\begin{proof}

Just as in the proof of Lemma \ref{easy}, we reduce to the case where $N=k$ in degree $0$ 
and can assume that $\deg H_0(M)$ is finite. 

 Now \begin{eqnarray} \nonumber 
\deg(\Tor^R_i(k,M))  &\leq&  \max_{a+b=i} \deg \Tor^{\bar{R}}_a\left( k, (\Tor^R_b(\bar{R}, M)\right)  
\\ & \leq_{(i)} &   \max_{a+b=i} \left( a \deg(\bar{R}) + \deg \Tor^R_b(\bar{R}, M)   \right) 
\\& \leq_{(ii)} &   \max_{a+b=i}  \left( a \deg(\bar{R}) + b A_1 + \delta(M)  \right) 
 \\ &  \leq_{(iii)} &  A_1 i + \max(\deg(H_0(M)),\deg(H_1(M))) + 4 A(R)
\label{eq:torboundwithnu}
\end{eqnarray}
where we used Lemma \ref{le:rbarbound} for step (i),  Lemma \ref{opm3lem} for step (ii),
and Lemma \ref{DLbound} together with the fact $\deg(\bar{R}) \leq A_1$ for step (iii).

\end{proof}

Having set up the basic bounds for the degrees over Tor groups over $R$ and $\bar{R}$, we are now ready to analyze the cohomology of the complex $\Koszul(M)$.
As the notation $\Koszul$ indicates, the complex $\Koszul$ is an analog of the Koszul complex, and the following result is an analog
of a well-known result in commutative algebra concerning Koszul complexes -- see, for example, Theorem 16.4 and following discussion
in \cite{Matsumura}.  

\begin{lem}  
 Each $H_q(\Koszul(R))$ is killed by the right action of $R_{>0}$.
\label{le:koszulhomologyispuny}
\end{lem}

\begin{proof} 

For $s = (h_1, \dots, h_n) \in S$, write $\partial s = h_1\cdots h_n \in G$, and define  \beq
S_g(g_0, \ldots, g_q; s ) = (g^{(g_0 \ldots g_q \partial s)^{-1}},g_0, \ldots, g_q; s)
\eeq
Extend $S_g$ linearly to $\Koszul(R)_{q+1} \to \Koszul(R)_{q+2}$.  By a routine computation, \beq
(S_g d + d S_g)(g_0, \ldots, g_q; s) = (g_0, \ldots, g_q;  r(g^{(\partial s)^{-1}}) s)
= (g_0, \ldots, g_q; s r_g) 
\eeq
and so ``right multiplication by $r_g$'' is homotopic to zero. \end{proof}

\begin{prop}  $\deg H_q(\Koszul(R)) \leq A_2 + q$, where $A_2 = \Atwo$.
\label{pr:K(r)}
\end{prop}

\begin{proof}

Multiplication by $U$ induces an endomorphism of the complex $\Koszul(R)$.  (Recall that $U$ is central, 
and thus it does not matter whether this multiplication is taken on the left or the right.)
Indeed, in the diagram \beq
\begin{CD}
\Koszul(R)_q @>d_q>> \Koszul(R)_{q-1} \\
@A{U}AA	@AA{U}A \\
\Koszul(R)_q @>d_q>> \Koszul(R)_{q-1} \\
\end{CD}
\eeq
the vertical arrows induce isomorphisms in source degree $n > A_1 + q $  (recall
that $\Koszul(R)_q$ is just a direct sum of copies of the shift $R[q]$.). This implies that the map
\beq
U: \ker d_q \ra \ker d_q
\eeq
is an isomorphism in source degree $n  >  A_1 + q $.  In particular, $\ker d_q$ is generated, as right $R$-module,  in degree at most $A_1  +  \deg U + q$, and the same is true for its quotient $H_q(\Koszul(R))$.  By Lemma~\ref{le:koszulhomologyispuny}, $H_q(\Koszul(R))$ is killed by $R_{>0}$; it follows that \beq
\deg(H_q(\Koszul(R))) \leq A_2+q,
\eeq
where $A_2 = A_1+\deg(U)$ is as in \eqref{constants}.
 \end{proof}

\begin{prop} \label{Koszulbound2}
Let $M$ be a left graded $R$-module.  Then
\beq
\deg H_q(\Koszul(M)) \leq \max(\deg(H_0(M)), \deg(H_1(M)) + A_1q  + (4A_1 +A_2).
\eeq
for all $q \geq 0$.
\end{prop}
\begin{proof} 
 We note that $\Koszul(M) = \Koszul(R) \otimes_R M$; then the ``universal coefficients'' spectral sequence: 
\beq
\Tor^R_i(H_{q-i}(\Koszul(R)), M) \Rightarrow H_q(\Koszul(M)).
\eeq
shows that $H_q(\Koszul(M))$ is filtered by subquotients of $\Tor^R_i(H_{q-i}(\Koszul(R)),M)$;
the degree of this $\Tor$-group is bounded by Proposition~\ref{pr:cd1} and Proposition ~\ref{pr:K(r)}:
$$ \deg \Tor^R_i(H_{q-i}(\Koszul(R)),M) \leq  \max(\deg H_0(M), \deg H_1(M)) + (A_2+q-i) +  A_1 i + 4A_1. $$

 This gives the result. 
\end{proof}

\subsection{Proof of Theorem \ref{Koszulbound}}
The last sentence has already been proved (Proposition \ref{pr:K(r)}).

We are going to show that
\begin{equation} \label{boundbykoszul} \deg H_i(M) \leq \deg H_i(\Koszul(M)) \ \ (i=0,1);\end{equation}
then (recall that $h_i = \deg H_i(\Koszul(M))$) the bound $h_q \leq \max(h_0, h_1) + (5A_1 +A_2) q$ 
follows from Proposition \ref{Koszulbound2}. Since we are taking $A_0 = 5A_1 + A_2$ (see \eqref{constants}) 
this will prove \eqref{theorem-equation} in the Theorem. 
  As for the the assertion about $M \stackrel{U}{\ra} M$, 
 Lemma
\ref{opm2lem} asserts that is an isomorphism in source degree 
at least $\delta(M) + A_1+1$; applying Lemma \ref{DLbound} and
\eqref{boundbykoszul} that number is 
  $$ \leq \max(h_0, h_1) + 4 A_1 + A_1+1.$$
Since $A_0 \geq 5 A_1 + 1$, 
we have proved that $M \stackrel{U}{\ra} M$ is an isomorphism in source degree at least
$\max(h_0, h_1) + A_0$ as desired.

Therefore, it remains to verify \eqref{boundbykoszul}.

The case $i=0$ of \eqref{boundbykoszul} follows
from the fact that  $H_0(M) = H_0(\mathcal{K}(M))$.
 
For the $H_1$ inequality, we will  factor the   map
$k[c] \otimes M[1] \rightarrow M$
from the final terms of the Koszul complex
\eqref{koszuldef} as follows:
$$k[c] \otimes_{k} M[1]  \stackrel{\alpha}{\rightarrow} R_{>0} \otimes_R M  \stackrel{\beta}{\rightarrow} M,$$
where the  first map  $\alpha$ sends $g \otimes m$ to $r_g \otimes m$ for $g \in c$, and
 the second map $\beta$ sends $i \otimes m$ to $im$.   Now $\alpha$ is surjective since elements $r_g$ generate $R_{>0}$ as a $R$-module. 
Also, $\alpha$ is degree-preserving
because the degree of $g \otimes m$ in $k[c] \otimes_R M[1]$ equals  $\deg(m)+1$
whereas the degree of $r_g \otimes m$ in $R_{>0} \otimes_R M$ also equals $\deg(m)+1$.

Let $\mathfrak{k}$ be the kernel of $ \beta \circ \alpha :k[c] \otimes  M[1] \rightarrow M$. Then there is a sequence
\beq
k[c^2] \otimes_{k} M[2] \stackrel{d }{\rightarrow} \mathfrak{k} \ra H_1(\mathcal{K}(M)) \ra 0
\eeq
which is exact at the middle and on the right.  Here $d$ is the differential from
the $\mathcal{K}$-complex  \eqref{koszuldef}. 

 It follows that $\mathfrak{k}$
is   generated as $k$-vector space by the image of $k[c^2] \otimes_{k} M[2] \stackrel{d }{\rightarrow} k[c] \otimes_{k} M[1]$, together
with terms in degree at most $\deg H_1(\mathcal{K}(M))$.  

But the composite map
$$ k[c^2] \otimes_{k} M[2] \stackrel{d}{\rightarrow} k[c] \otimes_{k} M[1] \stackrel{\alpha}{\rightarrow} R_{>0} \otimes_R M $$
is {\em zero}: it sends $(g_1, g_2) \otimes m$ 
first to $g_1 \otimes r(g_2) m - g_2 \otimes r(g_1^{g_2}) m$,
and then to $\left( r(g_1)r(g_2)  - r(g_2) r(g_1^{g_2}) \right) \otimes m$, 
which is zero in $R_{>0} \otimes_R M$.

Therefore  
$\alpha(\mathfrak{k})$ is spanned, as $k$-vector space
  by elements of degree $\leq \deg H_1(\mathcal{K}(M))$;
since $\alpha$  was surjective, that means that 
  $\ker(R_{>0} \otimes_R M \rightarrow M)$ is    supported in degree $\leq \deg H_1(\mathcal{K}(M))$.

But by tensoring the exact sequence $R_{>0} \rightarrow R \rightarrow k$
with $M$, we find an isomorphism $H_1(M) \simeq 
  \ker(R_{>0} \otimes_R M \rightarrow M)$. 
So we have shown 
 $\deg H_1(M) \leq \deg H_1(\mathcal{K}(M))$. This is the case $i=1$ of \eqref{boundbykoszul}. 
\qed

\section{The arc complex} \label{arc-complex}

In this section, we shall prove, as previously promised, that the homology of Hurwitz spaces
can be computed in terms of $\mathcal{K}$-complexes formed from homology of {\em smaller} 
Hurwitz spaces. 

Define the graded left $R$-module  
\begin{equation} \label{Mgluedef} 
M_p = \bigoplus_n  H_p(\Hur_{G,n}^c,k);\end{equation} 
the $R$-module structure arises from the gluing on Hurwitz spaces (\S \ref{gluing}),
and the grading is in the $n$-variable.    

\begin{prop} \label{hom-prop} 
There exists a homological spectral sequence $E^1_{qp}$ converging to $H_{q+p}(\Hur_{G,n}^c, k)$ in dimensions $q+p < n-2$.  Moreover, each row $(E^{1}_{*p}, d_1)$ is isomorphic
to the $n$th graded piece of $\Koszul(M_p)$, that is to say:
$$E^1_{qp} =\mbox{$n$th graded piece of } \Koszul(M_p)_{q+1} , \ \ p,q \geq 0, \ \ p+q<n-2.$$
\end{prop}

\subsection{}
This spectral sequence arises (in a similar way to \cite{Harer, HatcherWahl}) by considering the action of the braid group $B_n$ on (the geometric realization of) a highly connected simplicial complex $\mathcal{A}$ (a variant of the \emph{arc complex}), and filtering the complex by the dimension of simplices. 
We give the geometric construction of $\mathcal{A}$ in  \S \ref{arccomplex}.   

In the present
subsection, we give a combinatorial model of $\mathcal{A}$. This will be helpful for proofs. Namely, we construct a semisimplicial set $\mathbf{A}$ (i.e., we describe a set $\mathbf{A}_q$ of $q$-simplices for each $q$ and give consistent face maps $\partial_i: \mathbf{A}_q \to \mathbf{A}_{q-1}$), and it will follow from Proposition \ref{hw_prop} that in fact the geometric realization of $\mathbf{A}$ is   homeomorphic to that of $\mathcal{A}$.

Fix $n$. Let $L_q$ be the subgroup of $B_n$ (presented as in \eqref{artin-braid}) generated by 
$\sigma_{q+2}, \dots, \sigma_{n-1}$; it is abstractly isomorphic, then, to $B_{n - q-1}$. Note that if $q \geq n-2$ we understand $L_q$ to be the trivial group.  Let $\mathbf{A}_q = B_n/L_q$ (as a $B_n$-set).  Define the faces of the $q$-simplex $b L_q$ by the formula  
$$\partial_i(bL_q) = b s_{q,i} L_{q-1}, \; \; 0 \leq i \leq q,$$ 
 where $s_{q, i} = \sigma_{i+1} \sigma_{i+2} \cdots \sigma_q$, and we interpret $s_{q, q} = 1$.  Note that since $L_q$ actually commutes with $\sigma_j $ for $j \leq q$, this is independent of the choice of $b$ representing the coset $b L_q$.  

\begin{prop}

$\mathbf{A}$ is a semisimplicial set; that is, the semisimplicial identity $\partial_i \partial_j = \partial_{j-1} \partial_i$ holds for $i<j$.

\end{prop}

\begin{proof}

This is a computation in the braid group.  For $j=q$ it is the identity $b s_{q-1,i} L_{q-2} =b s_{q,i}  L_{q-2}$, which follows because $s_{q,i} = s_{q-1, i} \sigma_q$ and $\sigma_q \in L_{q-2}$.  For $j < q$:
\begin{eqnarray*}
\partial_i \partial_j (bL_q) & = & b \sigma_{j+1} \cdots \sigma_{q} \sigma_{i+1} \cdots \sigma_{q} L_{q-2} \\
 & = & b  \sigma_{i+1} \cdots \sigma_{j-1} \sigma_{j+1} \cdots \sigma_{q} \sigma_{j} \cdots \sigma_{q} L_{q-2} \\
 & = & b  \sigma_{i+1} \cdots \sigma_{j-1} \sigma_{j+1}\sigma_j \sigma_{j+2} \cdots \sigma_{q} \sigma_{j+1} \cdots \sigma_{q} L_{q-2} \\
 & =_{(a)} & b  \sigma_{i+1} \cdots \sigma_{j+1} \sigma_{j} \sigma_{j+1}^{-1} \sigma_{j+2} \cdots \sigma_{q} \sigma_{j+1} \cdots \sigma_{q} L_{q-2}. %
\end{eqnarray*}
where at step (a) the braid relation was used in the form $\sigma_{j+1} \sigma_j = (\sigma_j \sigma_{j+1} \sigma_j) \sigma_{j+1}^{-1}$.
Define, for $m=j, \dots, q-1$ the element
$$x_m := \sigma_{i+1}  \cdots \sigma_{m+1} \sigma_{j} \cdots \sigma_m \sigma_{m+1}^{-1} \underbrace{ \sigma_{m+2} \cdots \sigma_{q}} \sigma_{m+1} \cdots \sigma_{q}.$$
where the underbraced product is understood to be empty in the case when $m=q-1$.
The defining relations in the braid group yield the recursion $x_m = x_{m+1}$, so
$$\partial_i \partial_j (bL_q) = b x_j L_{q-2} = bx_{q-1} L_{q-2} =  b  \sigma_{i+1} \cdots \sigma_{q} \sigma_{j} \cdots \sigma_{q-1} L_{q-2} = \partial_{j-1} \partial_i (bL_q)$$
as desired.

\end{proof}

It may be helpful to note that the set of vertices of the $q$-simplex $b L_q$  is then given by $$b L_0, b \sigma_1 L_0, b \sigma_2 \sigma_1 L_0, \dots, b \sigma_q \dots \sigma_1 L_0.$$
Here, the $j$th vertex is obtained from the iterated face map 
$$b \sigma_j \dots \sigma_1 L_0 = \partial_0 \circ \dots \circ \partial_{j-1} \partial_{j+1} \circ \dots \circ \partial_q(bL_q).$$
One can check that the lone vertex missing in the face $\partial_j(bL_q) = b s_{q,j} L_{q-1}$ is indeed $b \sigma_j \dots \sigma_1 L_0.$

 \begin{prop*}  The geometric realization of $\mathbf{A}$ is $(n-2)$-connected. \end{prop*}
This  is a combinatorial reformulation of a  result of Hatcher-Wahl. We give the proof in \S \ref{arccomplex}.

Now consider the complex $\mathbf{A} \times c^n$, considered with the product $B_n$-action. 
Recall that for any topological space $Z$ endowed with a $B_n$-action, we write $Z\modmod B_n$ for the Borel construction $EB_n \times_{B_n} Z$; the homology of this space is the $B_n$-equivariant homology of $Z$.

It follows from the Proposition that the natural map:
$$  H_{p}((\mathbf{A} \times c^n)\modmod B_n) \ra   H_p(c^n\modmod B_n)  = H_p(\Hur_{G,n}^c) $$
is an isomorphism in degrees $p< n-2$.  
Note further that 
$$H_* ((\mathbf{A}_q \times c^n) \modmod  B_n) \cong H_*(c^n \modmod L_{q}) \cong H_*(c^{q+1} \times \Hur_{G, n-q-1}^c),$$
so $H_p((\mathbf{A}_q \times c^n) \modmod  B_n)$ is identified with $\Koszul(M_p)_{q+1}$.
 
Now, $(\mathbf{A} \times c^n)\modmod B_n$ is filtered by the simplicial structure on $\mathbf{A}$.  The resulting spectral sequence is of the form: 
$$H_p((\mathbf{A}_q \times c^n) \modmod  B_n) = E^1_{qp} \implies H_{p+q}((\mathbf{A} \times c^n)\modmod B_n),$$
and the target is isomorphic to $H_{p+q}(\Hur_{G, n}^c)$ when $p+q< n-2.$
 
\begin{lem} The first differential $d_1$ in the spectral sequence $E_{qp}^1 = H_p(c^n \modmod L_{q})$ is the alternating sum  
$$d_1 = \sum_{i=0}^q (-1)^i \partial_i = \sum_{i=0}^q (-1)^i [s_{q, i}^{-1}]$$ 
of the maps $[s_{q,i}^{-1}]: H_*(c^n \modmod L_{q}) \rightarrow H_*(c^n \modmod L_{q-1})$ induced by
$s_{q,i}^{-1}: c^n \rightarrow c^n$:
$$ s_{q,i}^{-1} : (g_0, \dots, g_q, g_{q+1}, \dots, g_{n-1})  \mapsto 
(g_0, \dots, \widehat{g_i}, g_{i+1}, \dots, g_{q}, g_i^{g_{i+1} \dots g_q}, g_{q+1}, \dots, g_{n-1}).$$
\end{lem}

Compare with \eqref{koszuldef}  to get Proposition \ref{hom-prop}. 

\begin{proof}

The face maps in the semi-simplicial space $EB_n \times_{B_n} (\mathbf{A} \times c^n)$ are induced by those in $\mathbf{A}$, as above.  Identifying $\mathbf{A}_q = B_n/ L_q$, this becomes $\partial_i: EB_n \times_{L_q} c^n \to EB_n \times_{L_{q-1}} c^n$ given by $$\partial_i(e, \mathbf{g}) = (es_{q,i} , \ s_{q,i}^{-1} \mathbf{g}),$$  where $0 \leq i \leq q$ and $\mathbf{g} \in c^n$.  

Since $L_q$ is a subgroup of $B_n$, we have a natural identification $H_*(c^n \modmod L_{q}) = H_*(EB_n \times_{L_q} c^n)$.  Since $s_{q, i}$ and $L_q$ commute, the maps $EB_n \times_{L_q} c^n \to EB_n \times_{L_{q-1}} c^n$
$$(e, \mathbf{g}) \mapsto (es_{q,i} , \ s_{q,i}^{-1} \mathbf{g}) \; \mbox{ and } \; (e, \mathbf{g}) \mapsto (e, \ s_{q,i}^{-1} \mathbf{g})$$
are freely homotopic, giving the result. 

\end{proof}

\subsection{The arc complex.}  \label{arccomplex}

We now prove that the geometric realisation of the complex $\mathbf{A}$ defined in the previous section is indeed $(n-2)$-connected, by identifying it with a geometric construction (the ``arc complex'') of Hatcher and Wahl.

Let $\Sigma$ be, as in  \S \ref{hur_sec}, an $n$-punctured disc.  By an {\em arc} on $\Sigma$ we mean a smooth path $\tau$ in $D$ from $*$ to one of the punctures $P_j$  with the following properties:
\begin{itemize}
\item[-]  $\tau$ is a smooth embedding $\tau: [0,1] \rightarrow D$ satisfying $\tau(0) = * \in \partial \Sigma$ and $\tau(1) = P_j$. 
\item[-] $0$ is the only element of the domain carried to the boundary of $\Sigma$, and $1$ the only element carried into the set $\{P_i\}$ of punctures. 
\item[-]
 The derivative $\tau'(0)$ is not  tangent to the boundary. 
  \end{itemize}

The \emph{arc complex} $\mathcal{A}$ is the simplicial complex whose vertex set consists of all isotopy classes of such arcs, and whose faces are collections of such isotopy classes which have representatives that intersect only at $*$, where they have distinct tangent vectors.   We denote by $\mathcal{A}_q$  the $q$-simplices of $\mathcal{A}$, i.e.  the set of isotopy classes of $(q+1)$-tuples of arcs intersecting only at $*$. 

We may define a partial ordering on the set of vertices: two arcs are comparable if they span an edge; then the ordering of the pair is given by the counterclockwise ordering of their tangent vectors at $*$.  This extends naturally to a total ordering on the set of vertices spanning a face; in this way, $\mathcal{A}$ becomes an ordered simplicial complex. 
This ordering equips the collection of sets $\mathcal{A}_*$ with the structure of a semisimplicial set: 
The face maps
$d_i: \mathcal{A}_q \to \mathcal{A}_{q-1} \ ( i =0, \dots, q)$ are defined by
\beq
d_i(\gamma_0, \dots, \gamma_q) = (\gamma_0, \dots, \widehat{\gamma_i}, \dots, \gamma_q).
\eeq
where $\gamma_0, \dots, \gamma_q$ are arranged in increasing order. 
In other words, the semisimplicial structure is given by deletion of arcs.

\begin{prop}  \label{hw_prop}
$\mathcal{A}$ is $(n-2)$-connected.  There exists an action of $B_n$ on $\arc$ which is transitive on $q$-simplices for each $q$.  Moreover, $\arc$ and $\mathbf{A}$ are $B_n$-equivariantly isomorphic. 
\end{prop}

\proof 

The connectivity assertion is a special case of Proposition 7.2 (and Definition 3.4) of Hatcher and Wahl \cite{HatcherWahl}. 
(In fact, $\mathcal{A}$ is contractible -- we refer the reader to the recent \cite{Damiolini} for a careful proof of this fact.)
Let  $\Diff_n$ be the group of orientation-preserving diffeomorphisms of $D$ which fix  the boundary pointwise and fix the set of punctures (setwise).
Then $\Diff_n$ acts on the set of isotopy classes of arcs on $\Sigma$ (i.e., the vertices of $\mathcal{A}$).  In fact, $\Diff_n$ acts on the whole of the complex $\mathcal{A}$, since diffeomorphisms preserve the non-intersection condition.  This evidently descends to an action of $B_n = \pi_0(\Diff_n)$ on $\mathcal{A}$. 

Now the  transitivity of $B_n$, as well as the final assertion -- which amounts to a computation of
the stabilizer of a simplex, see \eqref{eqvtisoRufus} below -- are very similar to \cite[Proposition 2.2]{Wahl}, which proves precisely the same result
in the context that the {\em endpoints} of all the arcs also coincide.   The proof of \cite[Proposition 2.2]{Wahl} also applies in this context.
For completeness, we  will  recall the main steps of this proof below.

\medskip

{\em Transitivity:} 
  We will be particularly brief about transitivity, since it is the easier part.
Given two $q$-simplices $\overline{\tau}$ and $\overline{\tau}'$, we choose systems of $q+1$ non-intersecting arcs $(\tau_0, \dots, \tau_q)$
and $(\tau_0', \dots, \tau_q')$ representing them.   While the proof of \cite[Proposition 2.2]{Wahl} transcribes
almost verbatim to this context,  we will sketch an alternate direct argument:  

Replacing the $\tau_i, \tau_i'$ by isotopic curves, and using the fact that the ordering of $\tau_i, \tau_i'$ coincide near $*$,
we may suppose that  all  $\tau_i$ and $\tau_i'$ are linear in a small neighbourhood of $0$  (i.e., straight near $*$), 
and that  $\tau_i(t) = \tau_i'(t)$ for $0 \leq t \leq   \varepsilon$ and some $\varepsilon  > 0$.  
Applying a suitable diffeomorphism, and using the fact that the braid group $B_n$ surjects to $S_n$, we may   assume that $\tau_i(1) = \tau_i'(1)$ for each $i$. 
Fix a smooth  increasing function $h$ on $[0,1]$ such that $h(t) = t $ for $ t \leq 2 \varepsilon/3$
and $h(1) = \varepsilon$.  Then $\tau_i \circ h = \tau_i' \circ h$. 

It is now sufficient
to show that there exists a ``arc-retracting'' diffeomorphism, i.e. a diffeomorphism $F$ which carries $\tau_i$
to the arc $\tau_i \circ h$  while fixing all punctures $P_i$ which are not of the form $\tau_i(1) \ \ (0 \leq i \leq q)$.
Once this is done, and a similar diffeomorphism $F'$ constructed for the $\tau_i'$, then $F' \circ F^{-1}$ gives the desired diffeomorphism
carrying $\tau_i$ to $\tau_i'$ and fixing all remaining punctures. 
  Let $R$ be the rectangle
$[\varepsilon/2, 1.01] \times [-1,1] $. Choosing a normal vector field to each $\tau_i$,  
we may find a collection of embeddings $g_i:  R\rightarrow D$,
 carrying $(t, 0)$ to $\tau_i(t)$ for $\frac{\varepsilon}{2} \leq t \leq 1$,  and such that the $g_i(R)$ are pairwise disjoint as well as disjoint from all other punctures -- in other words,
  $g_i(R)$ is an  explicit tubular neighbourhood of $\tau_i|_{[\varepsilon/2, 1]}$.   Our assertion is reduced to the claim that 
 there exists a diffeomorphism $G: R \rightarrow R$, trivial in a neighbourhood of the boundary, 
 and carrying  $(x,0)$ to $(h(x), 0)$ for
 $\varepsilon/2 \leq x \leq 1$.   It is routine to write down such a diffeomorphism explicitly.

  \medskip
  
  {\em Computation of stabilizer:} 
Order the punctures $P_1, \dots, P_n$ in such a way that the straight line segments $[*, P_i]$ from $*$ to $P_i$ are in counter-clockwise order around $*$. The {\em standard} $q$-simplex $v_q$ is the one consisting of linear arcs from $*$ to $P_1, \dots, P_{q+1}$.  We claim that the stabilizer of $v_q$ in $B_n$ is $L_q$.

Granting that, we define a $B_n$-equivariant semisimplicial map $\mathbf{A} \to \mathcal{A}$ on $q$-simplices by the bijection
\begin{equation} \label{eqvtisoRufus} B_n/L_q \to \mathcal{A}_q \; \mbox{ given by } \; b \mapsto b\cdot v_q.\end{equation} 
To see that this is semisimplicial, use the fact that $s_{q, i} v_{q-1} = d_i(v_q)$ (both consist of the $q$ straight line segments from $*$ to $P_1, \dots, \widehat{P_{i+1}}, \dots , P_{q+1}$).  The Proposition then follows. 

Therefore it remains only to check that the stabilizer of $v_q$ in $B_n$ is $L_q$.   Now, $L_q$ is generated by the Dehn twists $\sigma_i$ ($i \geq q+2$) which involve only the punctures $P_i$ and $P_{i+1}$, so it is apparent that $L_q$ does in fact stabilize $v_q$.   Additionally, there is a map from the mapping class group $$\Mod(\Sigma \setminus v_q) \to \Mod(\Sigma) = B_n$$ which extends diffeomorphisms by the identity on a neighborhood of $v_q$.   Write $\Sigma'$ for the complement of $n-q-1$ points in $D$. Then $\Mod(\Sigma')$ is isomorphic to $B_{n-q-1}$.  Picking a diffeomorphism $\Sigma' \to \Sigma \setminus v_q$ defines an isomorphism $B_{n-q-1} \cong \Mod(\Sigma \setminus v_q)$.  In fact, we may choose this diffeomorphism so that the composition 
$$B_{n-q-1} = \Mod(\Sigma') \cong \Mod(\Sigma \setminus v_q) \to \Mod(\Sigma) = B_n$$
carries $\sigma_j \in B_{n-q-1}$ to $\sigma_{j+q+1} \in B_n$.  The image of this map is precisely $L_q$.  

It therefore suffices to show that if $b \in B_n$ stabilizes $v_q$, then $b$ is isotopic to a diffeomorphism which is in the image of this map $\Mod(\Sigma \setminus v_q) \to \Mod(\Sigma)$.  
This can be done as in   \cite[p. 552-553]{Wahl}: Take a diffeomorphism $\phi$ representing $b$. {\em A priori} $b$ fixes only the isotopy class of each arc in $v_q$.
We may replace $\phi$ by an isotopic $\phi'$  that fixes the arcs in $v_q$ pointwise, inductively using the isotopy extension theorem \cite{palais}.  We briefly summarize the key point of the proofs of the proof in our setting, closely following \cite{Wahl} but indicating the minor differences:

A direct application of the isotopy extension theorem implies that $\phi$ is isotopic to a diffeomorphism that fixes the (linear) arc $a_1$ from $*$ to $P_1$. 
Replace $\phi$ by this isotopic diffeomorphism. Now the arc $a_2$ is isotopic to $\phi(a_2)$, by assumption, but the image of this isotopy
need not be disjoint from $\phi(a_1)=a_1$, so one cannot simply proceed as in the first step.    Let $H: [0,1] \times [0,1] \rightarrow D$ be an isotopy,
so that $H(0, -)$ and $H(1, -)$ correspond to the arcs $a_2$ and $\phi(a_2)$ respectively, and $H(-, 0) = *, H(-, 1)=P_2$. Adjust $H$ to be tranverse to $a_1$.  The preimage $H^{-1}(a_1)$ of the arc $a_1$
under $H$ is a union of circles and intervals.   Note that these components intersect the boundary of $[0,1] \times [0,1]$ only along $[0,1] \times 0$.  This is a difference from the setting of \cite{Wahl}, and means that we can ignore the case of a ``non-empty union of intervals'' mentioned at the bottom of page 552, {\em loc. cit.}.  In the present setting, every component of $H^{-1}(a_1)$ other than $[0,1] \times 0$ is contained in the interior of $[0,1] \times [0,1]$ and is thus a circle.   So each component of $H^{-1}(a_1)$ bounds a disk in $[0,1] \times [0, 1]$; restricting $H$ to this disk defines an element of 
the relative homotopy group $\pi_2(\Sigma, a_1)$,   and the triviality of that group implies
that we can assume that $H$ can be replaced by $H'$ for which $(H')^{-1}(a_1) =(H')^{-1}(*)$. At this point, we can again use the isotopy extension theorem
to ensure that $\phi$ fixes $a_2$, and we proceed inductively. 

 \qed

\section{Homological stability for Hurwitz spaces}
 
 We now prove the main theorem of the paper, that the homology of Hurwitz spaces
 stabilize under the non-splitting condition; it is by now an easy consequence of the main
 results of the prior three sections.

\begin{thm} Suppose $(G,c)$ satisfies the non-splitting condition, and let $k$ be a field in which $|G|$ is invertible.
Then there exist constants $A,B$ (depending on $G$) such that the map 
\begin{equation} \label{firstassertion} 
U: H_p(\Hur_{G,n}^c,k) \ra H_p(\Hur_{G,n+\deg U}^c,k)
\end{equation}
(see   Lemma \ref{le:rurfinite} and start of \S \ref{Rhomology} for definition of $U$) is an isomorphism whenever $n > Ap + B$.

The same assertion holds  for the restricted maps   
\begin{equation} \label{secondassertion} U: H_p(\CHur_{G,n}^c,k) \ra H_p(\CHur_{G,n+\deg U}^c,k).\end{equation} 
\label{th:stability}
\end{thm}

 Our proof will actually give a range $n > A'p +B'$ with different constants for the restricted maps \eqref{secondassertion} of the second assertion, 
 but one can then simply replace $A$ by $\max(A, A')$ and similarly for $B$. 

\begin{proof}
 
Let $M_p$ be the graded $R$-module corresponding to the $p$th homology of Hurwitz spaces, taken with $k$-coefficients,  as defined
in \eqref{Mgluedef}.  %
We prove, by increasing induction on $p$, that  for all $q \geq 0$ we have
$$
\deg  H_q( \Koszul(M_p) ) \leq A_2 + A_0(3p+q)
$$
from which the result follows, for suitable $B$, by the second assertion in Theorem \ref{Koszulbound}.
For $p=0$, we have $M_0 = R$; the inductive assumption follows in this case from Proposition \ref{pr:K(r)}, since
\beq
A_2 + A_0(3p+q) = A_2 + A_0 q \geq A_2 + q.
\eeq
in this case.

Now suppose the statement holds for $p < P$.  Consider
the leftmost part of $\Koszul(M_P)$, i.e. \begin{equation} \label{star} M_P \stackrel{e}{\leftarrow} k[c] \otimes M_P[1]  \stackrel{f}{\leftarrow} k [c^2] \otimes M_P[2] 
.\end{equation} 
The map $e$ is an edge morphism in the spectral sequence of Proposition \ref{hom-prop}, whereas
$f$ is identified with the differential $d_1: E^1_{1P} \ra E^1_{0P}$. 
  More generally, in the spectral sequence for $\Hur_{G, n}^c$, we have for each $q > 0$ that $E^2_{q, p}$ is the $n$th graded piece of $H_{q+1}(\Koszul(M_p))$.

The inductive hypothesis
  implies that \eqref{star}
is exact at the middle and  left term in degrees greater than
\beq
A_2 + 3 A_0 P.
\eeq
To see this, we note that the inductive hypothesis ensures that for $j>1$, $E^2_{j, P+1-j}$, or in other words the $n$th graded piece of $H_{j+1}(\Koszul(M_{P+1-j}))$, vanishes in degrees above
\begin{equation} \label{debug1} 
  A_2 + A_0(3P+4-2j).
\end{equation} 
Thus, once $n > A_2 + 3 A_0 P$, there are no differentials $d_j$ for any $j>1$ going into or out of $E^2_{0P}$.  Thus, $E^2_{0P} = E^{\infty}_{0P}$ in the spectral sequence for $H_P(\Hur_{G,n}^c,k)$ in the range $n > A_2 + 3 A_0 P$. 

Further, for $j>0$, $E^2_{j, P-j} = 0$ vanishes in degrees above
\begin{equation}  \label{debug2}
A_2 + A_0(3P+1-2j).
\end{equation} 
So, once $n > A_2 + A_0(3P-1)$, we have that all the graded pieces of $H_P(\Hur_{G,n}^c,k)$ other than $E^{\infty}_{0P} = E^2_{0P}$ vanish.  Hence $e: \mathrm{coker}(f) = E^2_{0P} \rightarrow M_P$ is an isomorphism in degrees above $A_2 + 3 A_0 P$. 

In other terms, 
$$\deg H_0(\Koszul(M_P)), \ \deg H_1(\Koszul(M_P)) \leq A_2 + 3 A_0 p.$$

Now apply Theorem \ref{Koszulbound}; it implies that, for $q \geq 2$,   $H_q( \Koszul(M_P))$
vanishes in degrees strictly above $A_2 + 3A_0 p + A_0 q$, which is precisely the inductive hypothesis to be proved.

 Now we address the final assertion of the theorem -- namely, that the same result holds for the space of connected covers.
 
 For each subgroup $Q$ of $G$ which is generated by $c \cap Q$,    write $\Hur_{G,n}^{Q,c}$ for the union of connected components of $\Hur_{G,n}^c$
 whose global monodromy is exactly equal to $Q$ --   i.e.  arising from elements  $(x_1, \dots, x_n) \in c^n$ with $\langle x_1, \dots, x_n \rangle = Q$.
Note that we can identify $\Hur_{G,n}^{Q,c}$ with $\CHur_{Q,n}^{c \cap Q}$. 
 
 What we have proved so far in this theorem applies equally well with $G$ replaced by $Q$,
 for $(Q, c)$ is still non-splitting.  
 By increasing induction on the order of $|Q|$ -- and possibly increasing the constants $A, B$ --  we may suppose
 that 
 $$U_Q: H_p(\Hur_{G,n}^{Q,c},k) \rightarrow H_p(\Hur_{G,n+\deg(U)}^{Q,c},k)$$
 is an isomorphism for $n \geq A p +B$ for every proper subgroup $Q$. 
 Here $U_Q$ is defined similarly to $U$, but ``relative to $Q$,''
 that is to say $U_Q = \sum_{g \in c \cap Q} r_g^{D |g|}$. Note that
 on $H_p(\Hur_{G,n}^{Q, c})$ we have
 
 \begin{equation} \label{UUq} 
 U = \frac{|c|}{|c \cap Q|} U_Q \mbox{ modulo $\bigoplus_{Q' \supsetneq Q} H_p(\Hur_{G,n+\deg(U)}^{Q', c},k)$} 
 \end{equation}

 What we have already proved shows that $U: H_p(\Hur_{G,n}^{G,c},k) \rightarrow H_p(\Hur_{G, n+\deg(U)}^{G, c},k)$
is injective for $n \geq Ap +B$, so it remains to verify surjectivity in that same range. 
Take $x \in H_p(\Hur_{G,n +\deg(U)}^{G,c},k)$; we may 
 write $x=Uy$ for some $y \in H_p(\Hur_{G,n}^{c},k)$.
 Write $y = \sum_{Q} y_Q$ where $y_Q \in H_p(\Hur_{G,n}^{Q,c},k)$.  
 
 By an increasing induction on the size of $Q$, 
 we see that $y_Q  = 0$ if $Q \neq G$: This is obviously true when $|Q|=1$. 
 Next if $m  < |G|$ and we know the assertion 
  is true for all $|Q| < m <|G|$, then take $Q$ of size $m$.  By \eqref{UUq} we get
$$ 0 =  \mbox{$Q$-component of $Uy$} =  \frac{|c |}{|c \cap Q|} U_Q y_Q$$
and by inductive assumption $U_Q$ is an isomorphism.
Therefore $y_Q = 0$.

This induction on the size of $Q$ has shown that $y \in H_p(\Hur_{G,n}^{G,c},k)$, and we're done. 

 \end{proof}

In the arithmetic applications to follow, we will be concerned with the quotients $\Hur_{G,n}^c/G$ under the $G$-action introduced previously
(see page \pageref{Gactionpageref}, after Definition \ref{hur_defn}).  These spaces are easily seen to stabilize in homology as well.

\begin{cor}
Suppose $G$ is center-free, and $(G,c)$ satisfies the non-splitting condition.  Let $k,U,A,B$ be as in Theorem~\ref{th:stability}.

Then the map
\beq
U: H_p(\Hur_{G,n}^c/G,k) \ra H_p(\Hur_{G,n+\deg U}^c/G,k)
\eeq
is an isomorphism whenever $n > Ap + B$.    This restricts to an isomorphism $U: H_p(\CHur_{G,n}^c/G,k) \ra H_p(\CHur_{G,n+\deg U}^c/G,k)$ in the same range. 
\label{co:modgstability}
\end{cor}

\begin{proof}  Recall that the action of $G$ on $\Hur_{G, n}^c$ is induced by the $B_n$-equivariant action of $G$ on $c^n$ given by termwise conjugation.  This is evidently not free, so the same is true for the action of $G$ on $\Hur_{G, n}^c$.  However, it is apparent from this description that the stabilizers $G_x$ of points $x \in \Hur_{G, n}^c$ are locally constant.   Namely:  when $x$ lies on a component whose global monodromy is $H$, the stabilizer $G_x$ is the centralizer of $H$ in $G$.  In particular, since $G$ is center-free by hypothesis, the action of $G$ on $\CHur_{G.n}$ is free.  Thus the quotient $q: EG \times_G \Hur_{G, n}^c \to \Hur_{G, n}^c / G$ is a fibre bundle with fibre over the class of $x$ equivalent to $BG_x$.  Since the order of $G$ (and hence $G_x$) is invertible in $k$, $q$ is an isomorphism in $H_*(-, k)$, so $H_*(\Hur_{G, n}^c / G, k)$ is isomorphic to the $G$-coinvariants in $H_*(\Hur_{G, n}^c, k)$.

The operator $U$, considered as a class in $H_0(\Hur_{G,\deg U}^c)$, is fixed by the action of $G$ (at least if $U$ is chosen as in Lemma~\ref{le:rurfinite}).  It then follows from the $G$-equivariance of the gluing maps that the isomorphism in Theorem \ref{th:stability} is one of $k[G]$-modules.  Taking $G$-coinvariants on both sides yields the desired isomorphism. 

\end{proof}

\section{Homological stability for Hurwitz schemes} \label{log}

So far, we have considered Hurwitz spaces as topological spaces parametrizing continuous branched covers of the disc.  In order to apply our theorems to arithmetic questions, we need to identify those topological spaces with the complex points of moduli schemes  (``Hurwitz schemes'') defined over arithmetic bases.

  In order to apply the topological results of the first part of the paper to counting problems over finite fields, we will need to  show that the \'{e}tale cohomology of Hurwitz schemes is ``the same" in characteristic $0$ and characteristic $p$.  This would follow immediately if the Hurwitz schemes  were smooth and proper; since they are not proper, more work is required, 
  involving the use of a compactification whose boundary has normal crossings. 

 Throughout this section, $G$ denotes a finite group. In our  main application,  $G$ will be of the form $A \rtimes \Z/2\Z$, 
where $A$ is a nontrivial finite $\ell$-group, and $\Z/2\Z$ acts on $A$ by inversion.

 For the reader's convenience we summarize the various spaces  (more precisely, schemes over $\Spec(\Z)$ or schemes over $\Spec(\baseR), \baseR := \Z[\frac{1}{|G|}]$) to be introduced in the following table. For example, 
$\SConf_n$ denotes the configuration space of $n$ distinct points on $\P^1$, and $\PreHn_{G,n}$ will denote the Hurwitz space 
of tame $G$-covers of $\P^1$ (see \S \ref{hurnotn}) which are branched at $n$ points.   

\begin{tabular}{ |  c|   c|     c|}
\hline
space & configurations  of  ...   &     Hurwitz space of $G$-covers  \\ 
&                                                & with corresponding branching \\ 
\hline
$\SConf_n$   (\S \ref{hurnotn}) &  $n$ points on $\P^1$ & $\PreHn_{G,n}$  (\S \ref{hurnotn})    \\
$\AConf_n$  (\S \ref{hurnotn2}) &  $n$ points on $\A^1$ & $\Hn_{G,n}$  (\S \ref{hurnotn2})    \\
 $\SPConf_n$  (\S \ref{hurnotn3})  & $n$ {\em labelled} points on $\A^1$ & $\PHn_{G,n}$  (p.\pageref{PHNdef})     \\
 \hline 
\end{tabular}
\medskip

Note that a $G$-cover ``branched at $n$ points on $\A^1$'' may or may not be branched at $\infty$, i.e.
it may have either $n$ or $n+1$ branch points on $\P^1$, and the  associated Hurwitz spaces will be correspondingly disconnected.

  We write (e.g.) $X/\Spec \Z$ or simply $X/\Z$ to denote a scheme $X$ equipped with its canonical morphism to $\Spec(\Z)$. 
For such a scheme, we will denote by $X/\F_q$ the scheme $X \times_{\Spec(\Z)} \Spec(\F_q)$, together with the natural morphism
to $\Spec(\F_q)$.

\subsection{} \label{hurnotn}

Our basic reference for Hurwitz schemes is the paper of Romagny and Wewers~\cite{roma:wewersromagny}.

If $P_n/\Spec \Z$ is the projective $n$-space with coordinates $a_0:a_1:\ldots:a_n$, which we think of as parametrizing  binary forms $a_0 X^n + a_1 X^{n-1} W+ \ldots + a_n W^n$ up to scaling, we define $\SConf_n / \Spec \Z$ to be\footnote{ The prime in $\SConf_n$
is because it represents a configuration space for the {\em projective} line,  whereas our previous $\Conf_n$ is homotopy equivalent to a configuration space for 
the affine line.}
the open subscheme of $P_n$ where the discriminant $\Delta(\sum_{i=0}^n a_i X^{n-i} W^i)$ doesn't vanish.  Then, for any field $k$, $\SConf_n(k)$ is the set of squarefree $k$-rational degree-$n$ divisors on $\P^1$.

Let $k$ be a field.   Then by a {\em tame $G$-cover of $\P^1$ over $k$} we shall mean a triple $(X,f,\phi)$, where
\begin{itemize}
\item $X$ is a smooth proper geometrically connected curve $X/k$;
\item $f: X \ra \P^1$ is  {\em tame}:  that is, there exists a  reduced divisor $D$ on $\P^1$ such that $f$ is \'{e}tale over $\P^1 - D$, and such that the ramification of $f$ over each geometric point of $D$ is nontrivial and prime to the characteristic of $k$;
\item $f$ is {\em Galois}:  that is, $\Aut(f)$ acts transitively on the geometric fibers of $f$;  \item $\phi$ is an isomorphism from $G$ to $\Aut(f)$.
\end{itemize}

We say a tame $G$-cover has $n$ branch points if $n$ is  the degree of the   unique divisor $D$ satisfying the condition in the second item above, and we call such a $D$ the {\em branch locus} of $f$.

Romagny and Wewers~\cite[Theorem 4.11]{roma:wewersromagny} construct a scheme $\PreHn_{G,n}$ with the following properties:
\begin{itemize}
\item $\PreHn_{G,n}$ is a scheme over $\baseR = \mathbb{Z}[\frac{1}{|G|}]$, endowed with a finite \'{e}tale morphism  \beq
\pi: \PreHn_{G,n} \ra \SConf_n/ \Spec R
\eeq
\item For every algebraically closed field $k$ with characteristic prime to $|G|$, there is an  $\mathrm{Aut}(k)$-equivariant bijection between $\PreHn_{G,n}(k)$ and the set of isomorphism classes of tame $G$-covers of $\P^1$ over $k$ with $n$ branch points.  If $x \in \PreHn_{G,n}(k)$ corresponds to a $G$-cover $f$, the image of $x$ in $\SConf_n(k)$ is the point parametrizing the branch locus of $f$ in $\P^1$.  If $G$ is center-free, the two statements above hold for an arbitrary field $k$, not only algebraically closed fields.
\end{itemize}

In fact, Romagny and Wewers construct a Hurwitz scheme over $\Spec \Z$, but for our present purpose we only need the open subscheme lying over $\Spec R$.  

The above assertions are all included in the statement of \cite[Theorem 4.11, Corollary 4.12]{roma:wewersromagny}, with the exception of the finiteness of $\pi$, which is \cite[Remark 4.15 ii]{roma:wewersromagny}.  Because the proof is not given in full there, we explain it briefly.   Each geometric fiber of the map from $\PreHn_{G,n}$ to $\SConf_n/R$,  say above the section $s: \Spec k \ra \SConf_n/R$ with $k$ algebraically closed, is in bijection with the set of $G$-covers $f$ over $k$  with a fixed branch locus $D \subset \P^1$ parametrized by $s$.  This set of $G$-covers, in turn, is in bijection with the set of conjugacy classes of surjective homomorphisms from $\pi_1^{et}(\P^1 - D, \bar{x})$ to $G$, where $\bar{x}$ is a geometric basepoint on $\P^1 - D$. Because $|G|$ is prime to the characteristic of $k$,  this number of surjections
is actually independent of the choice of the fiber $s$, see \cite[Expos{\'e} XIII, Corollaire 2.12]{SGA1}.   So all geometric fibers of $\pi$ have the same cardinality, whence $\pi$ is finite by 
\cite[Lemme 1.19, II]{DR}.

\begin{rem}
 A discussion of the functor represented by $\PreHn_{G,n}$, and also a construction of a representing stack in the case when $G$ has nontrivial center,   
is given in Wewers' (unpublished) thesis~\cite{wewers}   (see also   \cite[Cor 2.2]{roma:wewersromagny}).
We do not need this for our purposes. 
\end{rem}

\subsection{} \label{hurnotn2} 
The scheme $\PreHn_{G,n}$ is not exactly the right one for our purposes.  First of all, we want to study $G$-covers of $\A^1$, not of $\P^1$.  To this end, let $\AConf_n$ be the closed subscheme of $\SConf_{n+1}$ cut out by $a_0 = 0$; that is, we force the associated divisor to contain $\infty$ (i.e., the point $X=1$, $W=0$).
   $\AConf_n$ also embeds  as an open subscheme of $\SConf_n$ via
   \begin{equation} \label{sqf} [a_0: \dots : a_{n+1}] \mapsto [a_1: \dots : a_{n+1}]\end{equation}
   identifying it with the  open subscheme where the first coordinate is nonzero.

We now define $\Hn_{G,n}$ to be the disjoint union of $\PreHn_{G,n+1} \times_{\SConf_{n+1}} \AConf_n$ and $\PreHn_{G,n} \times_{\SConf_n} \AConf_n$, where the maps from $\AConf_n$ are the closed inclusion (respectively, the open inclusion) described above.  This somewhat convoluted definition is necessary because in our topological definition of Hurwitz spaces, the branched cover may be either ramified at $\infty$ or not,  these two cases corresponding to the two fiber products above.  In the language of the first section of the paper, the second component parametrizes those branched covers whose boundary monodromy is trivial.

Suppose that  $c$ is a rational union of conjugacy classes in $G$.   (Recall that a union of conjugacy classes is called {\em rational} when  
$g \in c \implies g^N \in c$ for each $N$ relatively prime to the order of $G$.) 
 We say a tame $G$-cover $f: X \ra \P^1$ has monodromy of type $c$ if the image of a generator of tame inertia at each branch point of $f$ lies in $c$.  (Because $c$ is rational, the choice of generator doesn't matter.)  Then there is a closed and open subscheme  $\Hn_{G,n}^c \subset \Hn_{G,n}$ parametrizing tame $G$-covers with monodromy of type $c$;  in the proof of \cite[Theorem 4.11]{roma:wewersromagny}, this appears as a disjoint union of some subset of the collection of schemes denoted there as $\mathcal{H}'_\mu$.

\begin{lem} \label{lem:comparison}  {Suppose that $G$ is center-free.} 
The complex manifold $\Hn_{G,n}(\C)$ is homeomorphic to the topological space $\CHur_{G,n}/G$ of Definition  \ref{hur_defn}; 
{similarly,   $\Hn_{G,n}^c(\C)$  is homeomorphic to $\CHur_{G,n}^c/G$. }
\end{lem}

\begin{proof}    See \cite[Theorem 4.11(iii)]{roma:wewersromagny}.  We explicate the map in our case:

 A $\C$-point of $\Hn_{G,n}(\C)$ is by definition a point of $\AConf_n$ -- that is to say  a subset $S \subset \C$ of size $n$ -- 
 together with  a tame  $G$-cover of $\mathbb{P}^1_{\C}$ {branched either at $S$ or at $S \cup \{\infty\}$.}
  
By comparison of {\'e}tale and topological $\pi_1$ \cite[Expos{\'e} XII, Theorem 5.1]{SGA1},  
and also using the same reasoning  used to construct a bijection of (b) and (c) of \S \ref{interps}, 
to give such a tame $G$-cover is the same as giving
 a conjugacy class of surjections $\pi_1(\mathbb{A}^1(\C)-S, x_0) \twoheadrightarrow G$.  Here, $x_0$ is an arbitrarily chosen basepoint.

Fix, once and for all, a homeomorphism of the interior of $D$ with $\mathbb{A}^1(\C)$. This allows
us to identify any subset $S \subset \C$ as before with a subset $S' \subset D$ of size $n$,
i.e. it induces an identification $\AConf_n(\C) \simeq \Conf_n$. 
Moreover, we have an identification, canonical up to conjugacy,
$$ \pi_1(\mathbb{A}^1(\C)- S, x_0) \simeq \pi_1(D-S', *).$$ 
 
 In particular, there is a canonical identification 
of conjugacy classes of surjections $f: \pi_1(\mathbb{A}^1(\C)-S, x_0) \twoheadrightarrow G$
with conjugacy classes of surjections $g: \pi_1(D-S', *) \twoheadrightarrow G$. 
For such a surjection $g$, the pair $(S', g)$ defines, as in \S  \ref{interps},  a point of $\CHur_{G,n}/G$: the quotient by $G$ arises from the fact that we have only a conjugacy class of surjections.

  This discussion has constructed a continuous function $\Hn_{G,n}(\C) \rightarrow \CHur_{G,n}/G$, 
covering the homeomorphism $\AConf_n(\C) \simeq \Conf_n$  made above. Moreover, 
this function induces a bijection between fibers of $\Hn_{G,n}(\C) \rightarrow \AConf_n(\C)$
and fibers of $\CHur_{G,n}/G \rightarrow \Conf_n$. Therefore it is a homeomorphism.
The map ``with monodromy $c$'' is obtained by restriction. 
  \end{proof}

\subsection{} \label{hurnotn3} 
 
Our basic tool for comparing cohomology in characteristic $0$ and characteristic $p$ is a suitable compactification of configuration space.
We denote by $\SPConf_n / \Spec \Z$ the complement, in $\A^n$, of the divisors $\set{z_i = z_j}$  (where $1 \leq i < j \leq n$). 
   This is a hyperplane complement whose points over any   ring $A$ are the ordered $n$-tuples $(z_1, \ldots, z_n)$ of sections in $\A^1(A)$ which are disjoint; that is, which satisfy $z_i - z_j \in A^\times$ whenever $i \neq j$. There is a map from $\SPConf_n$ to $\AConf_n$ sending $(z_1,\ldots,z_n)$ to the polynomial   $\prod_i (X-z_i W)$.

\begin{lem}   For all $n \geq 2$, $\SPConf_n$ has a compactification $X_n$ which is a smooth and proper  scheme over $\Spec \Z$, and such that the complement $D_n = X_n \bs  \SPConf_n$ is a relative normal crossings divisor.  
\label{le:xn}
\end{lem}

\begin{proof}  We remark first that the inclusion of  $\SPConf_n$ into $\P^n$ is a smooth proper compactification, but the complement is not normal crossings.

In the proof, we will make use of
the  moduli stack $\overline{\MM}_{0,n+1}$ of stable $(n+1)$-pointed genus-$0$ curves. 
It is a proper smooth scheme over $\Spec \Z$ and it contains the open subscheme $\MM_{0,n+1}$ corresponding  to nonsingular curves, and the complement
is a relative normal crossings divisor.   These facts are in
 Knudsen~\cite[Th 2.7]{knud:projectivityII}, except for the fact that $\overline{\MM}_{0,n+1}$ is a scheme. This follows from the identification (explained in the proof of {\em loc. cit.}, following from \cite[Prop 2.1]{knud:projectivityII}) of $\overline{\MM}_{0,n+1}$ with the universal curve over $\overline{\MM}_{0,n}$, together with the fact that $\overline{\MM}_{0,3} =\Spec(\Z)$ is a scheme.

There is a natural map from  $\SPConf_n \subset \A^n$ to $\MM_{0,{n+1}}$ sending a set of points $p_1, \ldots, p_n$ to 
$(\P^1; \infty, p_1, \dots, p_n)$, i.e. to the $n+1$-pointed curve defined by $\P^1$  together with the sections defined by $\infty, p_1, \ldots, p_n$.   

Let $\Aff$ be the group scheme of affine linear transformations of $\A^1$; that is, $\Aff$ is  the group scheme of upper triangular matrices $\left(\begin{array}{cc} A & B \\ 0 & 1 \end{array}\right)$,  corresponding to the transformation  $z \mapsto Az+B$.
 We can map  $\SPConf_n \subset \A^n$ to $\Aff$   via  %
 $(p_1, \ldots, p_n) \in \SPConf_n \mapsto (A, B) = (p_2-p_1, p_1)$. 
  
 The resulting (product) map $$F:  \SPConf_n \longrightarrow \MM_{0,n+1} \times \Aff$$ is an isomorphism:
For any base scheme $S$,  the $S$-points of $\SPConf_n$ are collections  $p_1, \dots, p_n \in \A^1(S) = \Gamma(S, \mathcal{O}_S)$, 
with the property that $p_i-p_j$ are units everywhere on $S$, and the map $F$ sends this to
$$(\P^1_S; \infty, p_1, p_2,  p_3, \dots, p_n) \in \MM_{0,n+1}(S) \times (p_2-p_1, p_1)  \in \Aff(S)$$ 
That this is a bijection follows from the fact that $\MM_{0,3} \simeq \Spec(\Z)$, which in particular  implies that if one fixes $u,v \in \Gamma(S, \mathcal{O}_S)$ with $u-v \in \Gamma(S, \mathcal{O}_S^{\times})$, 
any point  $\MM_{0,n+1}(S)$ is {\em uniquely representable}
by $(\P^1_S, \infty, u,v, q_3, \dots, q_n)$  for suitable sections $q_i$.

Since $\Aff$ is isomorphic, as scheme, to $\G_m \times \G_a$, it has a compactification by $\P^1 \times \P^1$ whose complement is normal crossings.  

So $X_n = \overline{{\MM}_{0,n+1}} \times \P^1 \times \P^1$ satisfies the requirements of the lemma.
\end{proof}

\begin{prop}
Let $A$ be a Henselian discrete valuation ring, whose quotient field has characteristic zero. Let $\bar{\eta}$ resp. $\bar{s}$ be a geometric generic (resp. special) point of $\Spec A$.   Let $X$ be a scheme proper and smooth over $\Spec A$, $D \subset X$ a reduced normal crossings divisor relative\footnote{Recall that this means that $D$ is  -- {\'e}tale locally on $X$ -- isomorphic to 
a union of coordinate hyperplanes in an affine space over $A$.} to $\Spec A$.

Let $U:= X - D$, and let $\pi: U' \ra U$ be a finite {\'e}tale cover. 
 Let $\GG$ be a finite group which acts compatibly on $U'$ and $U$, both actions
 covering the trivial $\GG$-action on $\Spec A$.

Then $H^i_{et}(U'_{\bar{\eta}},\Z/\ell \Z)$ and $H^i_{et}(U'_{\bar{s}},\Z/ \ell \Z)$ are isomorphic as $\GG$-modules for all $i$ and all primes $\ell$ invertible in $A$. 
\label{pr:logshit}
\end{prop} 

\begin{proof}
 
  We are grateful to a referee for providing a sketch of this argument,  which substantially simplifies our original proof. 
    
  Let $\mathcal{F} = \pi_*(\Z/\ell)$. It is a locally constant sheaf of $\Z/\ell$-modules on $U$ and it is tamely ramified \cite[Expos{\'e} XIII, 2.3(c)]{SGA1} along $D$:
 automatic because the generic point of $D$ has characteristic zero.
 
 Let $K$ be the sheaf $j_! \mathcal{F}$ on $X$.  Then $H^i(X_{\bar{s}}, K) = H^i_c(U_{\bar{s}}, \mathcal{F}) = H^i_c(U'_{\bar{s}}, \Z/\ell \Z)$,
  and exactly the same assertion holds with $\bar{s}$ replaced by $\bar{\eta}$.   The specialization 
  map (as a reference, see \cite[Chapter III, \S 3]{FK}) $H^i_c(U'_{\bar{s}},\Z/\ell \Z) \rightarrow H^i_c(U'_{\bar{\eta}}, \Z/\ell \Z)$
  is   $\mathcal{G}$-equivariant, by the functoriality of the constructions of {\em loc. cit.}, and 
  is identified with the corresponding specialization map for $X$ and $K$.    We prove it is an isomorphism:

The specialization map 
between cohomologies of special and generic fiber  fits into a triangle involving vanishing cycles (see \cite[Expos{\'e} XIII, equation 2.1.8.9]{SGA7}): 
$$H^i(X_{\bar{s}},  K) \rightarrow H^i(X_{\bar{\eta}},K) \rightarrow H^i(X_{\bar{s}}, \mathbf{R} \Phi_{\bar{\eta}} K) \stackrel{[1]}{\rightarrow} $$
where $H^i$ denotes hypercohomology in case of a complex of sheaves. The final term {\em vanishes} by \cite[Expos{\'e} XIII, 2.1.11]{SGA7}, and we conclude that the specialization map furnishes an isomorphism
$ H^i_c(U'_{\bar{s}}, \Z/\ell \Z) \rightarrow H^i_c(U'_{\bar{\eta}}, \Z/\ell \Z)$.

Because both $U'_{\bar{s}}$ and $U'_{\bar{\eta}}$ are smooth varieties, 
  Poincar{\'e} duality converts \eqref{new-iso} to the desired $\mathcal{G}$-equivariant  isomorphism of  usual (not compactly supported) cohomology:
\begin{equation} \label{new-iso} H^{j}(U'_{\bar{s}}, \Z/\ell \Z)  \stackrel{\sim}{\leftarrow} H^{j}(U'_{\bar{\eta}}, \Z/\ell \Z).\end{equation} 
  \end{proof}

With Proposition~\ref{pr:logshit} in hand we can now prove the desired upper bound for the Betti numbers of Hurwitz spaces over finite fields.  

\begin{prop}  Suppose $(G,c)$ satisfies the non-splitting condition, where $G$ is center-free and $c$ is a rational conjugacy class generating $G$.
 Then there exists $C(G,c)$ depending only on $(G,c)$ so that
\begin{equation}
\dim  H^i_{et}(\Hn_{G,n}^c / \bar{\F}_q, \Q_{\ell}) \leq C(G,c)^{i+1}
\label{ConnectedLinearRange} 
\end{equation}
for all $i,n$, so long as $\ell > \max(|G|, q, n)$.  
\label{pr:fqstability}
\end{prop}
We note that the restriction on $\ell$ is irrelevant for the application to Cohen-Lenstra heuristics in \S\ref{sec:CL}. 
\begin{proof}  
 Write $\PHn_{G,n}^c$ for the Cartesian product $\Hn_{G,n}^c \times_{\AConf_n} \SPConf_n$. \label{PHNdef}

Applying Proposition~\ref{pr:logshit}, together with the comparison
of {\'e}tale and analytic cohomology \cite[Theorem 4.4, Expos{\'e} XI]{SGA4},  to $A = W(\F_q)$, $U =  \SPConf_n$, $X = X_n$ as in Lemma~\ref{le:xn}, $U' = \PHn_{G,n}^c$, and $\GG = S_n$  acting by permuting points on $\SPConf_n$, we find that
$$H^i(\PHn_{G,n}^c(\C) ,\Z/\ell \Z) \cong H^i_{et}(\PHn_{G,n}^c/\Fqbar, \Z/\ell \Z) \ \ \mbox{(iso of $S_n$-modules)}$$ %
If $\ell > n$, the $S_n$-invariants on the mod $\ell$ cohomology of $\PHn_{G,n}^c$ recovers
the  mod $\ell$ cohomology of $\Hn_{G,n}^c$. Thus, supposing $\ell > n$, we obtain 
an isomorphism 
\begin{equation} \label{rfs} H^i(\Hn_{G,n}^c (\C),\Z/\ell \Z) \cong H^i_{et}(\Hn_{G,n}^c/\Fqbar, \Z/\ell \Z).\end{equation}

 Now  Lemma  \ref{lem:comparison} and   Corollary~\ref{co:modgstability}, together with an application of duality to pass between
 homology and cohomology, give a stability property for the left-hand side; thus we get the corresponding property for the right-hand side too:  
 \begin{equation} \label{isoc} H^p_{et}(\Hn_{G,n}^c/\bar{\F}_q, \Z/\ell \Z)  \cong H^p_{et}(\Hn_{G,n+D}^c/\bar{\F}_q, \Z/\ell \Z). 
\end{equation} 
whenever $n > Ap + B$ and $\ell > \max(q, |G|, n+D)$; recall that the constant $D$ was introduced in Lemma \ref{le:rurfinite}.  
Now \eqref{ConnectedLinearRange} with $\Z / \ell \Z$ coefficients in place of $\Q_\ell$ follows from Proposition \ref{Bettibound}.

  The statement with $\Q_{\ell}$-coefficients follows: The dimension of cohomology with $\Q_{\ell}$-coefficients
is bounded above by the rank (i.e., number of generators) of cohomology with $\Z_{\ell}$-coefficients, and the latter is bounded
above by the dimension of $\Z/\ell \Z$-cohomology by the universal coefficient sequence. 
\end{proof}

\begin{rem}  The requirement that $c$ generates $G$ is technically unnecessary:  if $c$ fails to generate $G$, then $\Hn_{G,n}^c$ is empty.
\end{rem}

  Note that we have {\em not} shown that  the isomorphism implicit in \eqref{isoc} is equivariant for the action of Frobenius on source and target.  Doing so would allow us to show that $\delta^+(q) = \delta^-(q)$ in Theorem~\ref{thm:CLweak}, so that we could talk about limits rather than limits inferior and superior.
    The reason for the deficit is that our stabilization map $U$ is constructed in an essentially non-algebraic way.   
  Although this problem can perhaps be remedied (cf. \cite[\S 4]{hainlooijenga}
  for the corresponding issue in the case of moduli spaces of curves), 
  we have not pursued this course in the present paper.    As we record in Conjecture~\ref{co:vanishing}, we believe that apart from ``obvious classes" both sides of the isomorphism in Proposition~\ref{pr:fqstability} are $0$, making the Frobenius equivariance vacuous.
    
\label{re:nofrob}

\section{The Cohen-Lenstra heuristics} \label{sec:CL}

The {\em Cohen-Lenstra heuristics} are a family of conjectures proposed by the two named authors~\cite{cohe:cohenlenstra} concerning the distribution of class groups of quadratic number fields among all finite abelian groups.  In fact, the phrase nowadays incorporates an even broader family of conjectures, worked out by Cohen, Lenstra, and Martinet \cite{cohenmartinet}, about class groups of number fields of all degree, with conditions on Galois group, and so forth.  
They make sense over any global field.

In this section, we shall prove Theorem \ref{thm:CLweak}, which sheds
some light on the Cohen-Lenstra heuristics over rational function fields over finite fields. 

\subsection{} 
\label{clintro}

Let $\lgp$ be the set of isomorphism classes of finite abelian $\ell$-groups. 

The {\em Cohen-Lenstra distribution} is a probability distribution on $\lgp$:
the $\mu$-mass of the (isomorphism class of) $A$ equals $$  \prod_{i \geq 1} (1-\ell^{-i}) \cdot |\Aut(A)|^{-1}.$$

The measure $\mu$ can be alternately described (see \cite{FW}) as the
distribution of the cokernel of a random map $\Z_{\ell}^{N} \rightarrow \Z_{\ell}^N$ ({\em random} according to the additive Haar measure on the space of such maps), as $N \rightarrow \infty$. 
From this latter description, we see that the expected number of surjections from
a $\mu$-random group into a fixed abelian $\ell$-group $A_0$ equals $1$.

In fact, this last remark {\em characterizes} $\mu$. {Writing $\Sur(G_1, G_2)$  for the set of surjections from the  group $G_1$ to the group $G_2$, we have:}
\begin{lem}
 If $\nu$ is any probability measure on $\lgp$ for which the expected
number of surjections from a $\nu$-random group to $A_0$ always equals $1$ -- i.e.,
$$ \sum_{B \in \lgp} \nu(B) \cdot  |\Sur(B, A_0)| = 1$$
for all $A_0 \in \mathcal{L}$ -- 
then
$\nu = \mu$.
\label{le:probmeasuremu}  
\end{lem}
\proof Indeed, the assumption gives, for every abelian $\ell$-group $A$,
\begin{equation} \label{assu} |\Aut(A)| \cdot \nu(A) + \sum_{B \in \lgp, B \neq A} |\SHom(B,A)| \cdot \nu(B) = 1. \end{equation}
Here $|\SHom(B,A)|$ denotes the number of surjections from (a representative for) $B$
to $A$.

\eqref{assu} forces, first of all, $$|\Aut(A)| \cdot \nu(A) \leq 1.$$  Inserting this upper bound
back into \eqref{assu} , we obtain the {\em lower} bound:
$$|\Aut(A)| \cdot \nu(A) \geq 1 - \beta, $$
where  $\beta :=\sum_{B \neq A} \frac{|\SHom(B,A)|}{|\Aut(B)|}  = 
\left( \prod_{i\geq 1} (1-\ell^{-i})^{-1} - 1 \right)$, the latter equality from the fact that \eqref{assu} holds for $\nu=\mu$.

Proceeding in this fashion, we find that $\nu(A) \cdot |\Aut(A)|$ is bounded
above and below by alternating partial sums of the series
$ 1 -  \beta + \beta^2 - \dots $
and consequently $\nu(A) \cdot |\Aut(A)| = \frac{1}{1+\beta}$ for every $A$, as desired. 
 \qed 
 
 This result admits a more quantitative form. 
 If $\nu$ is a probability
 measure on $\mathcal{L}$, and $A \in \mathcal{L}$, write
 $\langle \Sur(-, A) \rangle_{\nu}$ for the expected 
 number of surjections from a $\nu$-random group to $A$.

   \begin{prop} \label{cor:kluners}
  Suppose given $\epsilon_0  > 0$ and a finite subset $L \subset \lgp$. Then there exists $\delta > 0$ and a finite subset $L' \subset \lgp$ such that, if $\nu$ is any probability measure on $\lgp$ for which $\langle \Sur(-, A) \rangle_{\nu} \in [1-\delta, 1+\delta]$, for any $A \in L'$, then also $|\nu(A) - \mu(A)| \leq \epsilon_0$
  for any $A \in L$.  
   \end{prop}
  
  The proof will require the following:
  \begin{lem} \label{enhomlem}
     Given $\epsilon  > 0$ and $A \in \mathcal{L}$, 
      there exists a constant $c(A)$ and a finite subset
   $M \subset \mathcal{L}$ so that, whenever $|X| > c(A)$, 
   $$|\Sur(X,A)| \leq \epsilon  \frac{ \sum_{A' \in M} |\Sur(X,A')|}{|M|}.$$
   \end{lem}
   
  \begin{proof} 
  
     Call an {\em enlargement} of $A$  any group $A'$
 that admits a surjection $A' \twoheadrightarrow A$ with kernel of size $\ell$. 
We claim that for any abelian $\ell$-group $X$
surjecting onto $A$, with $|X| > |A|$,   there exists an enlargement $A'$ such that 
\begin{equation} \label{enhom} |\Sur(X, A')| \geq (\ell-1)  |\Sur(X, A)| .\end{equation}

Certainly there exists an enlargement $A'$ such that
$\Sur(X, A')$ is nonempty (take a suitable quotient of $X$). Fix such an $A'$, and fix a surjection $\pi: A' \twoheadrightarrow A$ and $f \in \Sur(X, A)$. 
We examine  lifts $\tilde{f}:X \rightarrow A'$ of $f$   (with respect to $\pi$). Then \eqref{enhom} follows from the fact that the number of such {\em surjective} lifts $\tilde{f}$ is at least $\ell-1$: 

\begin{itemize}
\item[(i)] 
If $A'$ is not isomorphic to $A \times \Z/\ell\Z$, any such lift $\tilde{f}$ is surjective, 
and the set of lifts
is a principal homogeneous space under $\Hom(X, \Z/\ell \Z)$. 

\item[(ii)] If $A'$ is isomorphic to $A \times \Z/\ell\Z$:
  Note that $\ker(f)$ cannot be contained in $\ell X$; if it were, 
 then $f$ induces an isomorphism $X/\ell \rightarrow A/\ell$, 
 but then the $\ell$-rank of $X$ and $A$ coincide, and 
 then $\Sur(X, A')$ would be empty, contradicting our choice of $A'$. 
 
 Thus there exists a homomorphism
 $\varphi: X \rightarrow \Z/\ell \Z$ that is nontrivial on $\ker(f)$,
 and $(f, \varphi)$ gives a surjection $X \rightarrow A'$ that lifts $f$.
 Since there are at least $(\ell-1)$ choices for $\varphi$, there
 are at least $\ell-1$ surjective lifts of $f$. 
 \end{itemize}  

We now iterate \eqref{enhom}: 
    
   Call an $s$-enlargement of $A$
   any group $A'$ that admits a map $A' \twoheadrightarrow A$ with kernel of size $\ell^s$. 
   Then, for any abelian $\ell$-group $X$ surjecting onto $A$ and of size larger than $\ell^s |A|$,
   we see that there exists an $s$-enlargement $A'$ so that
   $$ |\Sur(X, A')| \geq (\ell-1)^s |\Sur(X, A)| .$$
   On the other hand, the number of (isomorphism classes of) $s$-enlargements is bounded
   by the number of partitions of $s+m$, where $\ell^m  = |A|$.   Since
   $p(s+m) (\ell-1)^{-s} \rightarrow 0$ as $s \rightarrow \infty$, the statement
   of the Lemma follows, taking $M$ to be the set of $s$-enlargements of $A$ and $c(A) = \ell^s |A|$.
  \end{proof}

  We will now prove Proposition \ref{cor:kluners}. 
Recall that we say a sequence of measures $\nu_k$ on $\mathcal{L}$
  weakly converges to a limit $\nu_{\infty}$ if 
one has convergence of integrals $\int f \nu_k \rightarrow \int f \nu_{\infty}$ for each compactly supported continuous function;    for measures on the discrete space $\mathcal{L}$,  this is equivalent to asking that 
  $\nu_k(A) \rightarrow \nu_{\infty}(A)$, for every $A \in \mathcal{L}$.   Any sequence $\nu_k$ has 
   (by a diagonal argument) a weakly convergent subsequence.  However, the limit need not be a probability measure; it may assign $\mathcal{L}$ a mass that is strictly less than $1$. 
  \proof

  Let $L'_k$ be the subset of $\mathcal{L}$
  comprising groups with $|A| \leq k$. 
If the assertion were false, there is some measure $\nu_k$ that ``does not work''
for $L' = L'_k, \delta = 1/k$, that is to say:
  \begin{enumerate}
  \item $ | \langle  \Sur(-, A)_{\nu_k}  \rangle -1 | \leq 1/k$ for all  $A \in L'_k$;
  \item $|\nu_k(A) - \mu(A)| > \epsilon_0$ for some $A \in  L$. 
  \end{enumerate} 
Passing to a weakly convergent subsequence, 
we obtain measures $\nu_k$ having the following property:
\begin{equation} \label{Contrad} \lim_{k \rightarrow \infty} \langle \Sur(-, A) \rangle_{\nu_k}=1, \end{equation} 
for every fixed $A \in \mathcal{L}$, but $\nu_k$ weakly converge to a measure
$\nu_{\infty} \neq \mu$.  We will deduce a contradiction.

Fix $\epsilon > 0$ arbitrary.  This is not related to $\epsilon_0$ in the statement above: Indeed, we will apply Lemma~\ref{enhomlem} with this value of $\epsilon$,
and then let $\epsilon$ approach zero at the end of the argument.

We claim $\langle \Sur(-, A) \rangle_{\nu_{\infty}} = 1$:
 indeed, this expectation is $\leq 1$ by Fatou's lemma; on the other hand, with $c = c(A)$
 as in the statement of Lemma \ref{enhomlem}, 
\begin{eqnarray} \nonumber 
\langle \Sur(-, A) \rangle_{\nu_{\infty}} & = &   \sum_{|B| \leq c} \nu_{\infty}(B) |\Sur(B,A)|  +
 \sum_{|B| > c} \nu_{\infty}(B) |\Sur(B, A)|  \\  \nonumber
& \geq &  \lim_{k}  \sum_{|B|  \leq c} \nu_{k}(B) |\Sur(B,A)|    \\ \label{added_ref}
 & = &  1 - \lim_{k} \sum_{|B| > c} \nu_{k}(B) |\Sur(B,A)|  
 \end{eqnarray}
 
By Lemma~\ref{enhomlem}  \begin{equation} \label{added_ref_2}
\sum_{|B| > c} \nu_{k}(B) |\Sur(B,A)| \leq \epsilon |M|^{-1} \sum_{|B|>c,A' \in M} \nu_k(B) |\Sur(B,A')|
\end{equation} 
Now, by assumption, for any $A' \in M$ and any $k > |A'|$, 
\beq
\sum_{|B| > c} \nu_k(B) |\Sur(B,A')| \leq \langle  \Sur(-, A')  \rangle_{\nu_k} \leq 1+1/k
\eeq
and using \eqref{added_ref_2} and passing to the limit, we get  
\beq
\limsup_{k}  \sum_{|B| > c} \nu_{k}(B) |\Sur(B,A)| \leq \epsilon. 
\eeq
Thus  by \eqref{added_ref} and the prior discussion, we get 
$\langle \Sur(-, A) \rangle_{\nu_{\infty}}  \in [1-\epsilon, 1]$. Since $\epsilon$ was arbitrary, 
 \beq
\langle \Sur(-, A) \rangle_{\nu_{\infty}} = 1.
\eeq

Applying this conclusion with $A$ trivial, we see that $\nu_{\infty}$ is a probability measure; now Lemma~\ref{le:probmeasuremu} shows $\nu_{\infty} = \mu$, a contradiction. 
 \qed 
   
\subsection{}

The Cohen-Lenstra heuristics in the simplest case -- as formulated in \cite{cohe:cohenlenstra} -- assert that, for $\ell \neq 2$, the $\ell$-part of class groups of imaginary quadratic extensions of $\mathbb{Q}$ -- when ordered by discriminant -- approach $\mu$ in distribution.  
Precisely: amongst the set $S_X$ of imaginary quadratic fields of discriminant less than $X$, 
the fraction for which the $\ell$-part of the class group is isomorphic to $A$
approaches $\mu(A)$, as $X \rightarrow \infty$. 

In view of what we have just proved, this is equivalent to the validity of the following assertion for all abelian $\ell$-groups $A$:
the {\em average number of surjections} from the class group of a varying quadratic field to $A$
equals $1$.   Explicitly, 
\begin{equation} \label{cl2} 
\lim_{X \ra \infty} \frac{ \sum_{K \in S_X} \left| \Sur({\rm Cl}_K, A) \right| }{|S_X|}  = 1.\end{equation}

In the formulation \eqref{cl2}, there
are results for specific $A$:
\eqref{cl2} is true for $A = \Z/3\Z$ by work of Davenport and Heilbronn; the corresponding
assertion is even known over an arbitrary global field by work of Datskovsky and Wright~\cite{dats:dawr}; and a natural variant for $A = \Z/4\Z$ is a theorem of Fouvry and Kl\"{u}ners~\cite{fouv:4rank}.  

 If we replace $\mathbb{Q}$ by $\mathbb{F}_q(t)$, there is a fair amount of work (\cite{Achter, Washington}) 
 on the different problem (with no obvious analog over a number field) in which we
 fix the discriminant degree and take a limit as $q \ra \infty$.

  \subsection{}
  Let $\F_q$ be a finite field, and let $K = \F_q(t)$.  Let $\ell$ be an odd prime not dividing $q$,
  let $A$ be a nontrivial finite abelian $\ell$-group, and define
  $$ G := A \rtimes (\Z/2\Z),$$ where  the nontrivial element of $\Z/2\Z$ acts on $A$ by inversion.   Let $c$ be the conjugacy class in $G$ consisting of all involutions.  Then $c$ generates $G$.  Moreover, the pair $(G,c)$ is non-splitting by Lemma~\ref{le:dihedralnon-splitting}. 
     For brevity, we write simply $\mathbf{X}_n$ for the Hurwitz scheme
$\Hn_{G,n}^c \times_{\Spec(R)} \Spec \mathbb{F}_q$ that parameterizes  (\S \ref{hurnotn}, \S \ref{hurnotn2}) tame
$G$-covers of the affine line,  branched at $n$ points,  all of whose ramification is of type $c$.

\begin{prop}  Let $n$ be an odd integer.
  There is a bijection between $\mathbf{X}_n(\F_q)$ and the set of isomorphism classes of pairs $(L,\alpha)$, where $L$ is a quadratic extension  of $K$ of discriminant degree $n+1$ ramified at $\infty$, and $\alpha$ is a surjective homomorphism
\beq
\alpha: \Cl_L \ra A.
\eeq
Here, $\Cl_L$ is the class group of $\OO_L$, the integral closure of $\F_q[t]$ inside $L$.    Two pairs $(L, \alpha), (L', \alpha')$ are isomorphic if there exists
a $K$-isomorphism $f: L \rightarrow L'$ with $f^* \alpha' = \alpha$.

\label{pr:hypjac}
\end{prop}

  Note that $(L, \alpha)$ and $(L, \beta)$ are isomorphic if and only if $\alpha, \beta$
are interchanged by the automorphism of $L/K$, i.e. (see discussion below) if and only if $\beta = \pm \alpha$. 
 
Note that, under the assumptions, $\F_q$ is automatically algebraically closed inside $L$. 
If we denote by $C_L$ the smooth proper curve over $\F_q$ associated to $L$, 
then $\Cl_L$ is identified with the group of degree zero $\F_q$-rational divisors on $C_L$ up to
equivalence:  an ideal of $\OO_L$ defines a divisor  $D_I$ on $C_L$, 
and then $I \mapsto D_I - \mathrm{deg}(D_I). \infty$ descends to the desired isomorphism, 
where $\infty$ denotes the unique closed point of $C_L$ above $\infty$ on $\P^1$.  Also   $\Cl_L = \Cl(\OO_L)$ is   isomorphic to  the $\F_q$-points of the  Jacobian of $C_L$ (see e.g. \cite[Theorem C, (ii)]{Rosen}).  

\begin{proof} 
  
 For $L$ a quadratic extension of $K$ as in the proposition statement, we let 
  $\sigma$ be the nontrivial automorphism of $L$ over $K$, and let 
$J_L$ be the Galois group of the maximal abelian everywhere unramified extension  $E/L$
with pro-$\ell$ Galois group.  

Note that $E/K$ is Galois, with $\Gal(E/L) \simeq J_L$ as a normal subgroup. 
The subgroup $\langle x + \sigma(x) : x \in J_L \rangle$
is a normal subgroup of $\Gal(E/K)$. Let $F_L$ be the fixed field of this subgroup. 
  The extension $F_L/K$ is also Galois,  and its Galois group fits in an extension
\beq
1 \ra J_L' \ra \Gal(F_L/K) \ra \Gal(L/K) \ra 1
\eeq
where $J_L'$ is the quotient of $J_L$ by all elements $x + \sigma(x)$ with $x \in J_L$, i.e.
$J_L'$ is the largest quotient group of $J_L$ on which $\sigma$ acts by $-1$.
 
 Since $\Gal(L/K) \cong \Z/2\Z$ and $\ell \neq 2$, this sequence splits as a semidirect product
\begin{equation} \label{flk}
\Gal(F_L/K) = J_L' \rtimes \langle \tau \rangle
\end{equation} 
where $\tau$  {is any involution in $\Gal(F_L/K)$. } %

Class field theory yields a short exact sequence: 
\begin{equation} \label{geo-cft}  (\Cl_L)_{\ell} \hookrightarrow J_L  \twoheadrightarrow \widehat{\Z}_{\ell}\end{equation} 
where we have written $ (\Cl_L)_{\ell} $ for the Sylow $\ell$-subgroup
of $\Cl_L$.     Now $\sigma$ acts compatibly on all terms of \eqref{geo-cft}; the action on $\widehat{\Z}_{\ell}$ is trivial 
and its action on   $\Cl_L$ is by negation:  for any fractional ideal $I$ of $\OO_L$,
the product $I \cdot \sigma(I)$ is the extension of an ideal from $K$, and thus principal.
 
 Thus, the sequence \eqref{geo-cft} induces, via the snake lemma, a canonical isomorphism
between $J_L / (1 + \sigma) J_L$ and $(\Cl_L)_{\ell}/ (1+\sigma)  (\Cl_L)_{\ell}$ or, in other words,
\begin{equation} \label{geo-cft-2} J_L' \stackrel{\sim}{\rightarrow}  (\Cl_L)_{\ell}.\end{equation}

Therefore, beginning with $(L, \alpha)$ as in the statement of the theorem, 
we obtain by composing $\alpha$ with \eqref{geo-cft-2}  a surjection $f_{\alpha}: J_L' \twoheadrightarrow  A$;
by \eqref{flk} this extends to $g_{\alpha}: \Gal(F_L/K) \ra G$, by sending $\tau$ to any involution in $G$.
Since all involutions in $G$ are conjugate under $A$, the extension $g_{\alpha}$ is unique up to $A$-conjugacy.

Let $F_{\alpha}$ be the fixed field of $\ker(g_{\alpha})$; it is a Galois extension of $K$, equipped
with an isomorphism $\Gal(F_{\alpha}/K) \simeq G$ that is defined up to $A$-conjugacy.  { Moreover, $\F_q$ is algebraically closed inside $F_{\alpha}$. }
  To say the same geometrically: given $(L, \alpha)$ we have associated
  a  geometrically connected curve $Y_{\alpha}$ (namely, the curve associated to $F_{\alpha}$)   together with a map $Y_\alpha \ra \P^1$ and an isomorphism $g_{\alpha}:  \Aut(Y/\mathbb{P}^1) \rightarrow G$.

 The ramified  { places} of $F_{\alpha}/K$ are the same as that of $L/K$, because $F_{\alpha}/L$ is everywhere unramified. 
Finally, any  inertia group $I_v$ above such a ramified  place $v$  satisfies $g_{\alpha}(I_v) = \langle g \rangle$ for some $g \in c$:
indeed, $g_{\alpha}(I_v)$ is generated by a single element of $G$ which maps to the nontrivial element of $\Z/2\Z$. 

Although $g_{\alpha}$ is only well-defined up to $A$-conjugacy,
the  isomorphism class of the $G$-cover of $\mathbb{P}^1$  defined by $(Y_{\alpha}, g_{\alpha})$
{\em does not change} if we conjugate $g_{\alpha}$ by $A$.  In other words, we have defined a tame $G$-cover of $\P^1/\F_q$ branched at $n$ points of $\A^1$,  and its ramification is all of type $c$. By the description of the Hurwitz scheme in ~\S\ref{hurnotn}, this is equivalent to a point of $\mathbf{X}_n(\F_q)$.
  
Therefore our discussion yields a map 
\begin{equation} \label{gro-gro} \{ (L, \alpha) \} \mbox{ up to isomorphism} \longrightarrow \mathbf{X}_n(\F_q).\end{equation} 
 
 This map is  bijective: Let $X \rightarrow \mathbb{P}^1$ be a tame $G$-cover of $\P^1$, branched at $n$ points of $\mathbb{A}^1$. The quotient map $X \rightarrow X/A$ is etale
above $\mathbb{A}^1$ -- this follows from a local computation, using the fact  that the ramification is of type $c$.  Therefore, the degree $2$ quotient map $X/A \rightarrow \mathbb{P}^1$
is ramified at an odd number of points of $\mathbb{A}^1$, and must therefore
also be ramified at $\infty$. 
The monodromy of $X \ra \P^1$ above $\infty$ is then a cyclic subgroup that projects surjectively on $\Z/2\Z$; such a subgroup must be of order $2$, and so $X \rightarrow X/A$ is etale everywhere.  
 Thus we get an inverse to \eqref{gro-gro} by  sending this tame $G$-cover to  the pair $(L, \alpha)$, where $L$ is the function field of $X/A$ and $\alpha: \Cl_L \rightarrow A$ is the map
 arising from applying class field theory to the {\'e}tale cover $X \rightarrow X/A$.
 \end{proof}
Let $\Quads_n$ be the set of quadratic extensions of $K$ of the form $L = K(\sqrt{f(t)})$, where $f(t)$ is a  squarefree polynomial of odd degree $n$.  
  To exhaust $\Quads_n$, it is sufficient to let $f$ range through a set of representatives for squarefree polynomials
 up to the multiplication action of $(\mathbf{F}_q^*)^2$. The number of {\em monic} squarefree polynomials
 of degree $n$ with coefficients in $\mathbf{F}_q$ is equal \cite{Carlitz} to $q^n - q^{n-1}$, from where we deduce   \begin{equation} \label{Ihavebigquads} |\Quads_n| = 2 (q^n-q^{n-1}). \end{equation}

In what follows, we will average over fields in $\Quads_n$ and let $n \rightarrow \infty$.  
This is the analog of ``dyadic averages'' in analytic number theory, and we find it to be  aesthetically preferable in the case of a function field. 
However, it is easy to deduce similar results for averages  over sets such as  $\coprod_{m \leq n} \Quads_m$;
for example,  from Theorem \ref{th:weakcl} below, one immediately deduces the corresponding statement  
with $\Quads_n$ replaced by $\coprod_{m \leq n} \Quads_m$.

 Write $m_A(L) = |\Sur(\Cl(\OO_L), A)|$. 
Then -- in view of Proposition \ref{cor:kluners} -- the following theorem implies
Theorem \ref{thm:CLweak}.

\begin{thm}
\label{th:weakcl}
Let $\ell$ be an odd prime not dividing $q$ or $q-1$, and $A$ an $\ell$-group. 
There is a constant $B(A)$ such that
\begin{equation} \label{theorem-statement} 
\left| \frac{\sum_{L \in \Quads_n} m_A(L)}{|\Quads_n|}- 1 \right| \leq B(A)/\sqrt{q}
\end{equation}
for all $n,q$ with $\sqrt{q} > B(A)$ and $n$ an odd integer greater than $B(A)$.
\end{thm}

Here is the explicit argument that this implies Theorem \ref{thm:CLweak}:    
  Let $A_0$ be any fixed abelian $\ell$-group and let $\epsilon  > 0$. 
For a given $n$, let $\nu_n$ be the probability measure\footnote{We don't know that the limit $\lim_n \nu_n$ is a probability measure, but we are not using that.} on $\mathcal{L}$
with $\nu_n(A) $ equal to the fraction of $L \in \Quads_n$ with $\Cl(\OO_L) \simeq A$. 

Apply Proposition \ref{cor:kluners} to the measure $\nu_n$ and with $L  = \{A_0\}$;
the Corollary gives ``as output'' a finite list $L'$ of abelian $\ell$-groups and $\delta > 0$
with the property that 
$$  \mbox{ if   $\left| \frac{\sum_{L \in \Quads_n} m_A(L)}{|\Quads_n|}- 1 \right|  < \delta$ for all $A \in L'$ then $|\nu_n(A_0) - \mu(A_0)| < \epsilon$.  }$$
Now, notation  as in Theorem \ref{th:weakcl}, let $Q$ be chosen so that $Q > B(A)$ and 
$B(A)/\sqrt{Q} < \delta$ for every $A \in L'$;
we see that if $q > Q$ and $n   >Q$ we have
$   \left| \nu_n(A_0) -\mu(A_0) \right| \leq \epsilon.$
Thus, for any $q > Q$,   the upper and lower densities discussed in Theorem \ref{thm:CLweak} 
are both bounded between $\mu(A_0) - \epsilon$ and $\mu(A_0) + \epsilon$. 
Since $\epsilon$ was arbitrary, the result follows. 
 
\begin{proof}  
 
In the proof that follows, we use $H^i$ to denote $i$th {\'e}tale cohomology and $H^i_c$
to denote the corresponding compactly supported cohomology group. 
 
 Note that if $A$ is the trivial group, the left hand side of \eqref{theorem-statement} is zero.  
It  is enough to treat the case that $A$ is nontrivial.

By Proposition~\ref{pr:hypjac},  and the fact that $(L,\alpha)$ and $(L,\beta)$ are isomorphic if and only if $\alpha = \pm \beta$, we know that
\beq
  \sum_{L \in \Quads_n} m_A(L) =2  |\mathbf{X}_n(\F_q)|
\eeq
and as noted in \eqref{Ihavebigquads}, we have $|\Quads_n| = 2 (q^{n} - q^{n-1})$.
It will suffice, then, to show that
\begin{equation} \label{desid}  \left|  \frac{  |\mathbf{X}_n(\F_q)|}{q^n} - 1 \right|  \leq \frac{B(A)}{\sqrt{q}}.\end{equation} 
when $n$ and $q$ are sufficiently large relative to $A$.

Denote by $\bar{\mathbf{X}}_n$ the base change of $\mathbf{X}_n$ to $\bar{\F}_q$,  i.e. $\bar{\mathbf{X}}_n =   \Hn_{G,n}^c \times_{\Spec(R)} \Spec ( \overline{\mathbb{F}_q})$.   Fix a sufficiently large prime $\ell$. By \eqref{ConnectedLinearRange} and 
Poincar{\'e} duality for the smooth $n$-dimensional variety $\bar{\mathbf{X}}_n$ we get
 the existence of a constant $C(A)$  such that
\beq
\dim H^{2n-i}_c(\bar{\mathbf{X}}_n, \Q_{\ell}) = \dim H^i(\bar{\mathbf{X}}_n;\Q_\ell) \leq C(A)^{i}
\eeq
for all $i>0$ -- just take $C(A) = C(G,c)^2$ in the notation of \eqref{ConnectedLinearRange}.

  Deligne has proven \cite{DeligneWeil} 
that every eigenvalue of the geometric Frobenius  $\Frob_q$ (i.e., 
if we fix a projective embedding of $\mathbf{X}_n$ over $\F_q$,  
this is the operation which 
raises coordinates to the $q$th power) on {\em compactly supported} $H^j_c$ of a smooth  variety  %
is bounded above, in absolute value, by $q^{j/2}$.   Consequently, 
\begin{eqnarray*} 
\nonumber \left| q^{-n} \sum_{j < 2n} (-1)^j \Tr \left(  \Frob_q | H^j_{c}(\bar{\mathbf{X}}_n;\Q_\ell)  \right)  \right| & \leq & 
\nonumber q^{-n} \sum_{j=0}^{2n-1} q^{j/2} \dim H^j_{c}(\bar{\mathbf{X}}_n;\Q_\ell) \\ 
& \leq & \nonumber q^{-n} \sum_{j=0}^{2n-1}  C(A)^{2n-j} q^{j/2}   \\ & \leq &  \nonumber 
\sum_{k=1}^{\infty} \left( \frac{C(A)}{\sqrt{q}}  \right)^k. \\
\end{eqnarray*}
The last quantity above is at most $2\frac{C(A)}{\sqrt{q}}$ as long as $\frac{C(A)}{\sqrt{q}} \leq 1/2$; in other words, taking $B(A)$ to be $2C(A)$, we have 
\begin{equation}
\left| q^{-n} \sum_{j < 2n} (-1)^j \Tr \left(  \Frob_q | H^j_{c}(\bar{\mathbf{X}}_n;\Q_\ell)  \right) \right| \leq  \frac{ B(A)}{ \sqrt{q}}
\label{moo}
\end{equation}
whenever $\sqrt{q} > B(A)$.

We now claim that, for sufficiently large $n$ (this notion depending only on $A$) 
\beq
\Tr  \left( \Frob_q  | H^{2n}_c(\bar{\mathbf{X}}_n;\Q_\ell)  \right) = q^n
\eeq
By Poincar{\'e} duality, this is equivalent to the statement  that there is exactly one $\F_q$-rational connected component of $\bar{\mathbf{X}}_n$.
This, together with the Lefschetz trace formula \eqref{Gro-Lef}
and the bound \eqref{moo},  will imply \eqref{desid}. 
 
Let $\eta$ be the generic point of $\AConf_n \times_{\Spec R} \Spec \F_q$, and write $K$ for the function field of $\eta$.  Then the finite \'{e}tale cover  
$\mathbf{X}_n \ra \AConf_n \times_{\Spec R} \Spec \F_q$ is determined by its geometric generic fiber $\Sigma$ together with the action of $\Gal(\bar{K}/K)$ on that fiber.  The latter group sits in a sequence
\beq
\Gal(\bar{K}/\overline{\F_q} K) \ra \Gal(\bar{K}/K) \ra \Gal(\overline{\F_q}/\F_q)
\eeq

Then the desired conclusion -- that there is exactly one $\F_q$-rational   connected component of $\bar{\mathbf{X}}_n$ -- is precisely the statement that only one $\Gal(\bar{K}/\overline{\F_q} K)$-orbit on $\Sigma$ is preserved by the action of $\Gal(\overline{\F_q}/\F_q)$.

We prove this by expressing $\Sigma$ in a different way, allowing us to make contact with the existing literature on monodromy in families of hyperelliptic curves.
  
  Recall (\S \ref{hurnotn2})  the definition of $\AConf_n$ as the subscheme of $\SConf_{n+1}$ with $a_0 = 0$.    
    Let $C$
 be the smooth hyperelliptic curve over $K$ birational to the plane curve
 $$Y^2  =  a_1 X^{n} + a_2 X^{n-1} W +  \dots +  a_{n+1} W^n.$$
Choose $k$ sufficiently large so that $\ell^k A = 0$.     Let $V$ be the   $\ell^k$-torsion   points of the Jacobian $\Jac(C)$  over $\bar{K}$. 
Then $V  \simeq (\Z/\ell^k\Z)^{2g}$ (with $g = \lfloor \frac{n-1}{2} \rfloor$, which, under our standing hypothesis that $n$ is odd, equals $\frac{n-1}{2}$)  and we have a  {\em monodromy homomorphism} \beq
\mu:  \Gal(\bar{K}/K)  \ra \mathrm{Aut}(V). \eeq

Now consider the set $\Sur(V,A)$ of surjective homomorphisms from $V$ to $A$. This set carries a natural action of $\Gal(\bar{K}/K)$ derived from $\mu$.

Since $\AConf_n$ is a moduli scheme for degree-$n$ squarefree divisors on $\A^1$, there is a universal such divisor on $\A^1 / \AConf_n$, which restricts to a canonical degree-$n$ squarefree   (i.e., reduced) divisor $D$ on $\A^1/\bar{K}$.  The set $\mathbf{X}_n(\bar{K})$ of tame $G$-covers of $\A^1/\bar{K}$ branched at $D$ (which is to say $\Sigma$)  is naturally identified by the argument of Proposition~\ref{pr:hypjac} with $\Sur(V,A)$, equivariantly for the action of $\Gal(\bar{K}/K)$ on both sides.  (In the proof of Proposition~\ref{pr:hypjac}, we replace the statement of class field theory, which enters at \eqref{geo-cft}, with   the fact (see e.g. \cite[(2.4)]{KL}; in the case at hand, this is just Kummer theory) that  the abelian {\'e}tale extensions of $C / \bar{K}$ 
with Galois group $A$ are classified by  surjections $\Jac(C)[\ell^k] (\bar{K}) \twoheadrightarrow A$.)  %
 
It thus suffices to show  that only one $\Gal(\bar{K}/\overline{\F_q}K)$-orbit on $\Sur(V,A)$ is preserved by the action of $\Gal(\overline{\F_q}/\F_q)$ (again, for $n$ large enough). 
 
The action of $\Gal(\bar{K}/\overline{\F_q}K)$ on $V$ preserves the Weil pairing $V \times V \rightarrow \Z/{\ell^k}\Z(1)$, which
we write as $\langle v_1, v_2 \rangle$ for $v_1, v_2 \in V$. 
For $m \in (\Z/\ell^k)^{\times}$, write $\GSp_m(V)$  for all automorphisms $\alpha \in \Aut(V)$ that satisfy $\langle \alpha(v_1), \alpha(v_2) \rangle= m    \langle v_1, v_2 \rangle $;  write $\Sp(V) = \GSp_1(V)$ for automorphisms preserving $\langle -, - \rangle$,
and $\GSp(V) $ for $\bigcup_m \GSp_m(V)$.  Thus 
$$ \mu(\Gal(\bar{K}/K)) \subset \GSp(V) \mbox{ and } \mu(\Gal(\bar{K}/\overline{\F_q} K)) \subset \Sp(V).$$
Moreover, if $F$ is an element of $\Gal(\bar{K}/K)$ lying over Frobenius in $\Gal(\overline{\F_q}/\F_q)$, then $\mu(F)$   lies in $\GSp_q(V)$.
 
  Jiu-Kang Yu has proved~\cite{jkyu:monodromy} that $\mu(\Gal(\bar{K}/\overline{\F_q} K))=  \Sp(V)$ for large enough $g$ in this case (again using that $\ell \neq 2$; for $\ell = 2$ the monodromy group is in fact smaller.)\footnote{In fact, he proved this as part of a program to study the Cohen-Lenstra conjecture over function fields, just as we do; his theorem on monodromy allows him to prove a result in the $q \ra \infty$ limit as alluded to above.}  For other proofs of Yu's (unpublished) result, see Achter--Pries~\cite[Theorem 3.4]{achterpries} and Hall~\cite[Theorem 4.1]{cjh}.  
This ``big monodromy" theorem simplifies the situation considerably:
The geometric components of $\bar{\mathbf{X}}_n / \Fqbar$ are  therefore in bijection with $\Sp(V)$-orbits on $\Sur(V,A)$, and an orbit  $O$ is defined over $\F_q$ if and only if the stabilizer in $\GSp(V)$ of 
  some $x \in O$ (equivalently: every $x \in O$)   has nontrivial intersection with $\GSp_q(V)$.

  We claim that, for sufficiently large $n$, there is a {\em unique} $\Sp(V)$-orbit on $\Sur(V,A)$ defined over $\F_q$. 
    This  can be reduced to a corresponding ``linear algebra'' statement with $V$ replaced by $\Z_{\ell}^{2g}$ as follows: 
Write  $T$ for the full Tate module of $\Jac(C)$, so that $T/\ell^k \simeq V$.   Because we chose $k$ so that $\ell^k A = 0$, 
the pullback under $T \rightarrow V$ identifies $\Sur(V, A) \stackrel{\sim}{\rightarrow} \Sur(T, A)$.  By smoothness, 
the map $\GSp_q(T) \rightarrow \GSp_q(V)$ is surjective. 
    Therefore, our desired conclusion follows from the subsequent Lemma. 
   \end{proof}

 \begin{lem}   \label{sp-la}  
 {  Let $V$   be a finite free $\Z_{\ell}$-module of rank $2g$, equipped with a  perfect symplectic pairing $\omega :V \times V \rightarrow \Z_{\ell}$.}
 Let $A$ be a finite abelian $\ell$-group and $q \in \Z_{\ell}^{\times}$ be such  that $q-1$ is invertible in $\Z_{\ell}$.  Define
 $O$ as the set of all surjections $V \rightarrow A$ whose stabilizer, inside $\GSp(V)$, intersects $\GSp_q(V)$ nontrivially: 
\begin{equation} \label{oder} O=  \{ \mbox{$f : V \rightarrow A$ surjective, and there exists $h \in \GSp_q(V)$  with $f \circ h= f$} \} \end{equation}  
   Then, for $g$ sufficiently large, $O$ is nonempty, and $\Sp(V)$ acts transitively on $O$. 
  \end{lem}

\begin{proof}

 We shall use the following four facts, all of which remain valid for any finite rank free $\Z_{\ell}$-module  $M$ with a nondegenerate symplectic form 
(``nondegenerate'' means that the symplectic form induces an isomorphism $M \rightarrow \Hom_{\Z_{\ell}}(M, \Z_{\ell})$):

 \begin{itemize}
 \item[(i)] 
  Any two maximal isotropic $\Z_{\ell}$-submodules of $V$ are conjugate to one another under $\Sp(V)$.

 \item[(ii)] If a direct sum decomposition $V = V_1 \oplus V_2$ is orthogonal for $\omega$, the restriction $\omega|V_j$ is nondegenerate for $j=1,2$.
 
  \item[(iii)] $V$ admits a decomposition $V_+ \oplus V_-$, where both $V_+, V_-$ are maximal isotropic. 
 
   \item[(iv)] Given a decomposition $V = A \oplus B$, where both $A, B$ are isotropic for $\omega$, 
 then $A,B$ are maximal isotropic.

 \end{itemize} 

For (i) one can argue by 
extending a free $\Z_{\ell}$-basis $\{x_1, \dots, x_g\}$ for a maximal isotropic $\Z_{\ell}$-submodule
to a standard symplectic basis for $V$: by nondegeneracy, choose a basis $y_j$ with $\langle x_i, y_j \rangle = \delta_{ij}$, 
and then successively modify $y_j$ by a combination $\sum_{i \leq j} a_i x_i$ so that $\langle y_i, y_j \rangle = 0$. 
For (ii), we note that a ``degenerate vector'' in $V_1$, i.e. a vector $v_1 \in V_1$ that satisfies
$\langle v_1, w \rangle \in \ell \Z_{\ell}$ for all $w \in V_1$, would also be degenerate when considered as a vector in $V_1 \oplus V_2$. 
  (iii) is immediate from  Corollary 3.5  of  \cite[Chapter 1]{MH}.  (iv) follows from the corresponding fact for symplectic forms over fields.

 To ensure that $O$ is nonempty when $g$ is sufficiently large, write $V$ as the direct sum of two maximal isotropic subspaces $V_+ \oplus V_-$.  The automorphism of $V$ which acts as $q$ on $V_+$ and $1$ on $V_-$ lies in $\GSp_q(V)$, and fixes any surjection from $V$ to $A$ factoring through projection to $V_-$.  Such a surjection exists as long as $g \geq \dim_{\F_\ell} A/\ell A$.

  It remains to verify that $\Sp(V)$ acts transitively on $O$.

Take  $f \in O$; there exists $h \in \GSp_q(V)$ so that the image of $h-1$ is contained in the kernel of $f$.

 Let $V_1$ (resp. $V_q$) be the sum of generalized eigenspaces of $h$ on $V$
  for all eigenvalues  
that  reduce to $1$ (resp. $q$) in $\mathbb{\bar{F}}_{\ell}$.  
By this, we mean more precisely the following: set $\overline{V} = V \otimes \overline{\Q_{\ell}}$,
 and then set
 $$V_1 = V \cap \bigoplus_{|\lambda - 1|<1} \overline{V}_{\lambda},$$
where $\overline{V}_{\lambda}$ is the generalized $\lambda$-eigenspace; $V_q$ is defined similarly.
  Equivalently,
$V_1$ (resp. $V_q$) consists of all $v \in V$ for which $(h-1)^n v \rightarrow 0$ (resp. $(h-q)^n v \rightarrow 0$) as $n \rightarrow \infty$.  
 
Let $W$ be the sum 
of all other generalized eigenspaces of $h$ on $V$, i.e.  for all eigenvalues $\lambda$ that satisfy $|(\lambda-1)(\lambda-q)|=1$; in other words,
\begin{equation} \label{dsumdecomp}
W = \cap_{n=1}^\infty (h-1)^n (h-q)^n V.
\end{equation} 

 Then, since $q$ and $1$ are distinct in ${\mathbb{F}_{\ell}}$, we have  
\begin{equation} \label{dsumdecomp2} V = V_1 \oplus V_q \oplus W.\end{equation}
 Indeed, given $v \in V$, we can certainly write $v = v_1 + v_q + w$ where $v_1, v_q, w$ lie (respectively) in the $\Q_{\ell}$-spans
of $V_1, V_q, W$. By applying a large power of $(h-1) (h-q)$ we deduce that $(h-1)^n (h-q)^n w \in W$
for sufficiently large $n$.  But $(h-1) (h-q)$ is  easily seen to be invertible on $W$, so in fact $w \in W$. 
It follows then that $v_1 +v_q \in V$, and proceeding similarly we see $v_1 \in V_1, v_q \in V_q$, yielding  \eqref{dsumdecomp2}. 

 Moreover, $V_1 \oplus V_q$ is orthogonal to $W$, and both $V_1$ and $V_q$ are isotropic: 
to see this, let $x$ be an element of $V_1$ and $y$ an element of $V_1 \oplus W$.  Then $(h-1)^n x$ approaches $0$ as $n \ra \infty$, while for all $n$ there exists $z_n \in V$ such that $y = (h-q)^n z_n$.  Now
\begin{eqnarray*}
\omega(x,y) & = & \omega(x,h(h-q)^{n-1} z_n) - \omega(x,q(h-q)^{n-1}z_n) \\
& = & \omega(x,h(h-q)^{n-1} z_n) - \omega(hx,h(h-q)^{n-1}z_n) \\
& = & \omega((1-h)x,(h-q)^{n-1}x_n), \mbox{ with $x_n = h z_n$.}  \\
\end{eqnarray*}
Iterating, we see that $\omega(x,y)$ lies in $\omega((h-1)^n x, V)$; this being the case for all $n$, we have $\omega(x,y) = 0$.  The proof that $V_q$ is orthogonal to $V_q \oplus W$ is exactly the same.

By (ii) $W$ is nondegenerate. By (iii) we can express $W = W_{+} \oplus W_{-}$
as the sum of two isotropic submodules; since
$V = (W_+ \oplus V_1) \oplus (W_{-} \oplus V_q)$
and both summands are isotropic, they are by (iv) both maximal isotropic. 
In particular, $W_- \oplus V_q$ is a  maximal isotropic submodule of 
$V$,   which furthermore lies in the image of $h-1$, and thus belongs to $\mathrm{ker}(f)$.

Now fix a decomposition $V = V_+ \oplus V_-$ into isotropic submodules, both free of rank $g$ over $\Z_{\ell}$. Modifying $f$ by an element of $\Sp(V)$, we may assume  by (i) that $f$ factors through 
the projection
$V \longrightarrow V_+$.
Since every automorphism of $\GL(V_+)$ is induced by an element
of $\Sp(V)$, we are reduced to checking that any two surjections
$V_+ \rightarrow A$ are conjugate under $\GL(V_+)$.

We must show that any two surjections $f_1, f_2: \Z_{\ell}^g \rightarrow A$ are conjugate under $\GL_g(\Z_{\ell})$.
 Fix $x_1, \dots, x_k \in A$ such that the classes of $x_i$  form a basis for $A/\ell A$ as a $\Z/\ell$-vector space.  Then $x_1, \dots, x_k$ generate $A$. 
Lift $x_1, \dots, x_k$ to $y_1, \dots, y_k \in \Z_{\ell}^g$. 
The $y_i$ are linearly independent modulo $\ell$, and so 
we can extend $y_1, \dots, y_k$ to a   $\Z_{\ell}$-basis
$y_1, \dots,  y_g$ for $\Z_{\ell}^g$  where
$f(y_i) = 0$ for $i > k$: simply extend arbitrarily to a basis, and then modify the $y_i$'s for $i > g$ by linear combinations
of $y_1, \dots, y_k$ to ensure that $f(y_i)=0$ for $i  > k$. 
Similarly lift $x_1, \dots, x_k$ via $f_2$ to $y_1', \dots, y_g'$;   the automorphism of $\Z_{\ell}^g$
carrying $y_i$ to $y_i'$ then carries $f_1$ to $f_2$. 
\end{proof}

\bibliography{evw}

\def\cprime{$'$} \def\cprime{$'$}
\begin{thebibliography}{10}

\bibitem{SGA4}
{\em Th\'eorie des topos et cohomologie \'etale des sch\'emas. {T}ome 1:
  {T}h\'eorie des topos}.
\newblock Lecture Notes in Mathematics, Vol. 269. Springer-Verlag, Berlin-New
  York, 1972.
\newblock S{\'e}minaire de G{\'e}om{\'e}trie Alg{\'e}brique du Bois-Marie
  1963--1964 (SGA 4), Dirig{\'e} par M. Artin, A. Grothendieck, et J. L.
  Verdier. Avec la collaboration de N. Bourbaki, P. Deligne et B. Saint-Donat.

\bibitem{SGA7}
{\em Groupes de monodromie en g\'eom\'etrie alg\'ebrique. {II}}.
\newblock Lecture Notes in Mathematics, Vol. 340. Springer-Verlag, Berlin-New
  York, 1973.
\newblock S{\'e}minaire de G{\'e}om{\'e}trie Alg{\'e}brique du Bois-Marie
  1967--1969 (SGA 7 II), Dirig{\'e} par P. Deligne et N. Katz.

\bibitem{acv}
Dan Abramovich, Alessio Corti, and Angelo Vistoli.
\newblock Twisted bundles and admissible covers.
\newblock {\em Comm. Algebra}, 31(8):3547--3618, 2003.
\newblock Special issue in honor of Steven L. Kleiman.

\bibitem{Achter}
Jeffrey~D. Achter.
\newblock Results of {C}ohen-{L}enstra type for quadratic function fields.
\newblock In {\em Computational arithmetic geometry}, volume 463 of {\em
  Contemp. Math.}, pages 1--7. Amer. Math. Soc., Providence, RI, 2008.

\bibitem{achterpries}
Jeffrey~D. Achter and Rachel Pries.
\newblock The integral monodromy of hyperelliptic and trielliptic curves.
\newblock {\em Math. Ann.}, 338(1):187--206, 2007.

\bibitem{Arnold}
V.~I. Arnol{\cprime}d.
\newblock The cohomology ring of the colored braid group.
\newblock {\em Mat. Zametki}, 5:227--231, 1969.

\bibitem{Artin}
E.~Artin.
\newblock Theory of braids.
\newblock {\em Ann. of Math. (2)}, 48:101--126, 1947.

\bibitem{Borel}
Armand Borel.
\newblock Stable real cohomology of arithmetic groups.
\newblock {\em Ann. Sci. \'Ecole Norm. Sup. (4)}, 7:235--272 (1975), 1974.

\bibitem{Milgram}
C.~P. Boyer, J.~C. Hurtubise, and R.~J. Milgram.
\newblock Stability theorems for spaces of rational curves.
\newblock {\em Internat. J. Math.}, 12(2):223--262, 2001.

\bibitem{byeo:indivisibility}
Dongho Byeon.
\newblock Indivisibility of class numbers of imaginary quadratic function
  fields.
\newblock {\em Acta Arith.}, 132(4):373--376, 2008.

\bibitem{Carlitz}
L~Carlitz.
\newblock The arithmetic of polynomials in a {G}alois field.
\newblock {\em Proc. N. A. S}, pages 120--122, 1931.

\bibitem{Charney-Davis}
Ruth Charney and Michael~W. Davis.
\newblock Finite {$K(\pi, 1)$}s for {A}rtin groups.
\newblock In {\em Prospects in topology ({P}rinceton, {NJ}, 1994)}, volume 138
  of {\em Ann. of Math. Stud.}, pages 110--124. Princeton Univ. Press,
  Princeton, NJ, 1995.

\bibitem{Cohen}
Frederick~R. Cohen, Thomas~J. Lada, and J.~Peter May.
\newblock {\em The homology of iterated loop spaces}.
\newblock Lecture Notes in Mathematics, Vol. 533. Springer-Verlag, Berlin-New
  York, 1976.

\bibitem{cohe:cohenlenstra}
H.~Cohen and H.~W. Lenstra, Jr.
\newblock Heuristics on class groups of number fields.
\newblock In {\em Number theory, {N}oordwijkerhout 1983 ({N}oordwijkerhout,
  1983)}, volume 1068 of {\em Lecture Notes in Math.}, pages 33--62. Springer,
  Berlin, 1984.

\bibitem{cohenmartinet}
Henri Cohen and Jacques Martinet.
\newblock \'{E}tude heuristique des groupes de classes des corps de nombres.
\newblock {\em J. Reine Angew. Math.}, 404:39--76, 1990.

\bibitem{Damiolini}
Chiara Damiolini.
\newblock The braid group and the arc complex, 2013.

\bibitem{dats:dawr}
Boris Datskovsky and David~J. Wright.
\newblock Density of discriminants of cubic extensions.
\newblock {\em J. Reine Angew. Math.}, 386:116--138, 1988.

\bibitem{pace:pacellireu}
Michael Daub, Jaclyn Lang, Mona Merling, Allison~M. Pacelli, Natee Pitiwan, and
  Michael Rosen.
\newblock Function fields with class number indivisible by a prime $\ell$,
  2009.

\bibitem{dejongkatz}
A.~J. DeJong and N.~Katz.
\newblock Counting the number of curves over a finite field.
\newblock \url{http://www.math.columbia.edu/~dejong/papers/curves.dvi},
  Preprint, 2001.

\bibitem{DR}
P.~Deligne and M.~Rapoport.
\newblock Les sch\'emas de modules de courbes elliptiques.
\newblock In {\em Modular functions of one variable, {II} ({P}roc. {I}nternat.
  {S}ummer {S}chool, {U}niv. {A}ntwerp, {A}ntwerp, 1972)}, pages 143--316.
  Lecture Notes in Math., Vol. 349. Springer, Berlin, 1973.

\bibitem{DeligneWeil}
Pierre Deligne.
\newblock La conjecture de {W}eil. {I}.
\newblock {\em Inst. Hautes \'Etudes Sci. Publ. Math.}, (43):273--307, 1974.

\bibitem{Deligne}
Pierre Deligne.
\newblock La conjecture de {W}eil. {II}.
\newblock {\em Inst. Hautes \'Etudes Sci. Publ. Math.}, (52):137--252, 1980.

\bibitem{EV}
Jordan~S. Ellenberg and Akshay Venkatesh.
\newblock Counting extensions of function fields with bounded discriminant and
  specified {G}alois group.
\newblock In {\em Geometric methods in algebra and number theory}, volume 235
  of {\em Progr. Math.}, pages 151--168. Birkh\"auser Boston, Boston, MA, 2005.

\bibitem{LettertodeJong}
Jordan~S. Ellenberg and Akshay Venkatesh.
\newblock Letter to {A}. {J}. de {J}ong, 2008.

\bibitem{fouv:4rank}
{\'E}tienne Fouvry and J{\"u}rgen Kl{\"u}ners.
\newblock On the 4-rank of class groups of quadratic number fields.
\newblock {\em Invent. Math.}, 167(3):455--513, 2007.

\bibitem{FK}
Eberhard Freitag and Reinhardt Kiehl.
\newblock {\em \'{E}tale cohomology and the {W}eil conjecture}, volume~13 of
  {\em Ergebnisse der Mathematik und ihrer Grenzgebiete (3) [Results in
  Mathematics and Related Areas (3)]}.
\newblock Springer-Verlag, Berlin, 1988.
\newblock Translated from the German by Betty S. Waterhouse and William C.
  Waterhouse, With an historical introduction by J. A. Dieudonn{\'e}.

\bibitem{friedvolklein}
Michael~D. Fried and Helmut V{\"o}lklein.
\newblock The inverse {G}alois problem and rational points on moduli spaces.
\newblock {\em Math. Ann.}, 290(4):771--800, 1991.

\bibitem{FW}
Eduardo Friedman and Lawrence~C. Washington.
\newblock On the distribution of divisor class groups of curves over a finite
  field.
\newblock In {\em Th\'eorie des nombres ({Q}uebec, {PQ}, 1987)}, pages
  227--239. de Gruyter, Berlin, 1989.

\bibitem{frie:divisibility}
Christian Friesen.
\newblock Class number divisibility in real quadratic function fields.
\newblock {\em Canad. Math. Bull.}, 35(3):361--370, 1992.

\bibitem{gartonthesis}
Derek Garton.
\newblock Random matrices, the {C}ohen-{L}enstra heuristics, and roots of
  unity.
\newblock {\em Algebra and Number Theory}, 9(1):149--171, 2015.

\bibitem{SGA1}
Alexander Grothendieck.
\newblock {\em Rev\^etements \'etales et groupe fondamental. {F}asc. {II}:
  {E}xpos\'es 6, 8 \`a 11}, volume 1960/61 of {\em S\'eminaire de G\'eom\'etrie
  Alg\'ebrique}.
\newblock Institut des Hautes \'Etudes Scientifiques, Paris, 1963.

\bibitem{hainlooijenga}
Richard Hain and Eduard Looijenga.
\newblock Mapping class groups and moduli spaces of curves.
\newblock In {\em Algebraic geometry---{S}anta {C}ruz 1995}, volume~62 of {\em
  Proc. Sympos. Pure Math.}, pages 97--142. Amer. Math. Soc., Providence, RI,
  1997.

\bibitem{cjh}
Chris Hall.
\newblock Big symplectic or orthogonal monodromy modulo {$l$}.
\newblock {\em Duke Math. J.}, 141(1):179--203, 2008.

\bibitem{Harer}
John~L. Harer.
\newblock Stability of the homology of the mapping class groups of orientable
  surfaces.
\newblock {\em Ann. of Math. (2)}, 121(2):215--249, 1985.

\bibitem{HatcherWahl}
Allen Hatcher and Nathalie Wahl.
\newblock Stabilization for mapping class groups of 3-manifolds.
\newblock {\em Duke Math. J.}, 155(2):205--269, 2010.

\bibitem{hurwitz}
Adolf Hurwitz.
\newblock {\"U}ber {R}iemann'sche {F}l{\"a}chen mit gegebenen
  {V}erzweigungspunkten.
\newblock {\em Mathematische Annalen}, 39(1):1--60, 1891.

\bibitem{KL}
Nicholas~M. Katz and Serge Lang.
\newblock Finiteness theorems in geometric classfield theory.
\newblock {\em Enseign. Math. (2)}, 27(3-4):285--319 (1982), 1981.
\newblock With an appendix by Kenneth A. Ribet.

\bibitem{knud:projectivityII}
Finn~F. Knudsen.
\newblock The projectivity of the moduli space of stable curves. {II}. {T}he
  stacks {$M\sb{g,n}$}.
\newblock {\em Math. Scand.}, 52(2):161--199, 1983.

\bibitem{mall:clroots}
G.~Malle.
\newblock {Cohen--Lenstra heuristic and roots of unity}.
\newblock {\em Journal of Number Theory}, 128(10):2823--2835, 2008.

\bibitem{Matsumura}
Hideyuki Matsumura.
\newblock {\em Commutative ring theory}, volume~8 of {\em Cambridge Studies in
  Advanced Mathematics}.
\newblock Cambridge University Press, Cambridge, 1986.
\newblock Translated from the Japanese by M. Reid.

\bibitem{MH}
John Milnor and Dale Husemoller.
\newblock {\em Symmetric bilinear forms}.
\newblock Springer-Verlag, New York-Heidelberg, 1973.
\newblock Ergebnisse der Mathematik und ihrer Grenzgebiete, Band 73.

\bibitem{palais}
Richard~S. Palais.
\newblock Local triviality of the restriction map for embeddings.
\newblock {\em Comment. Math. Helv.}, 34:305--312, 1960.

\bibitem{Quillen}
Daniel Quillen.
\newblock Finite generation of the groups {$K\sb{i}$} of rings of algebraic
  integers.
\newblock In {\em Algebraic {$K$}-theory, {I}: {H}igher {$K$}-theories ({P}roc.
  {C}onf., {B}attelle {M}emorial {I}nst., {S}eattle, {W}ash., 1972)}, pages
  179--198. Lecture Notes in Math., Vol. 341. Springer, Berlin, 1973.

\bibitem{roma:wewersromagny}
Matthieu Romagny and Stefan Wewers.
\newblock Hurwitz spaces.
\newblock In {\em Groupes de {G}alois arithm\'etiques et diff\'erentiels},
  volume~13 of {\em S\'emin. Congr.}, pages 313--341. Soc. Math. France, Paris,
  2006.

\bibitem{Rosen}
Michael Rosen.
\newblock {$S$}-units and {$S$}-class group in algebraic function fields.
\newblock {\em J. Algebra}, 26:98--108, 1973.

\bibitem{Salvetti}
M.~Salvetti.
\newblock Topology of the complement of real hyperplanes in {${\bf C}^N$}.
\newblock {\em Invent. Math.}, 88(3):603--618, 1987.

\bibitem{Segal}
Graeme Segal.
\newblock The topology of spaces of rational functions.
\newblock {\em Acta Math.}, 143(1-2):39--72, 1979.

\bibitem{Wahl}
Nathalie Wahl.
\newblock Homological stability for mapping class groups of surfaces.
\newblock In {\em Handbook of Moduli, Vol. III}, volume~26 of {\em Advanced
  Lectures in Mathematics}, pages 547--583. International Press, Boston, MA,
  2013.

\bibitem{Washington}
Lawrence~C. Washington.
\newblock Some remarks on {C}ohen-{L}enstra heuristics.
\newblock {\em Math. Comp.}, 47(176):741--747, 1986.

\bibitem{wewers}
Stefan Wewers.
\newblock Construction of {H}urwitz spaces.
\newblock Thesis, U. Duisburg- Essen, 1998.

\bibitem{jkyu:monodromy}
Jiu-Kang Yu.
\newblock Toward a proof of the {C}ohen-{L}enstra conjecture in the function
  field case.
\newblock {\em preprint}, 1997.

\end{thebibliography}

\bibliographystyle{plain}

\end{document}